\documentclass[reqno,11pt]{amsart}
\usepackage{amssymb}
\usepackage{amsmath}
\usepackage{mathrsfs}
\usepackage[usenames]{color}
\usepackage{bm}

\numberwithin{equation}{section}

\allowdisplaybreaks

\begin{document}

\newtheorem{theorem}{Theorem}[section] 
\newtheorem{proposition}[theorem]{Proposition}
\newtheorem{corollary}[theorem]{Corollary}
\newtheorem{lemma}[theorem]{Lemma}

\theoremstyle{definition}
\newtheorem{assumption}[theorem]{Assumption}
\newtheorem{definition}[theorem]{Definition}

\theoremstyle{definition} 
\newtheorem{remark}[theorem]{Remark}
\newtheorem{remarks}[theorem]{Remarks}
\newtheorem{example}[theorem]{Example}
\newtheorem{examples}[theorem]{Examples}
\newenvironment{pf}%
{\begin{sloppypar}\noindent{\bf Proof.}}%
{\hspace*{\fill}$\square$\vspace{6mm}\end{sloppypar}}
\def\R{{\mathbb R}}
\def\N{{\mathbb N}}
\def\C{{\mathbb C}}
\def\Q{{\mathbb Q}}
\def\bR{{\mathbb R}}
\def\mP{{\mathbb P}}
\def\mQ{{\mathbb Q}}
\def\mL{{\mathbb L}}
\def\mD{{\mathbb D}}
\def\mA{{\mathbb A}}
\def\mB{{\mathbb B}}
\def\cR{{\mathcal R}}
\def\cS{{\mathcal S}}
\def\cQ{{\mathcal Q}}
\def\curl{{\mathrm{curl}\,}}
\def\div{{\mathrm{div}}}
\def\eps{{\varepsilon}}
\def\vphi{{\varphi}}
\def\om{{\omega}}
\def\la{{\lambda}}
\def\Om{{\Omega}}
\def\cT{{\mathcal{T}}}
\def\cH{{\mathcal{H}}}
\def\cO{{\mathcal{O}}}
\def\cI{{\mathcal{I}}}
\def\hW{{\widehat W}}
\def\sLis{{mathscr L}_{is}}
\def\hook{{\,\hookrightarrow\,}}
\def\hookd{\,{\stackrel{_d}{\hookrightarrow}\,}}
\def\vp{{\varphi}}
\def\la{{\langle}}
\def\ra{{\rangle}}

\hyphenation{Lipschitz}

\sloppy
\title[Regularity for the Stokes equations on a 2{D} wedge domain]
{Optimal Sobolev regularity for the Stokes equations\\ on a 2{D} wedge
domain}

\author[M.\ K\"ohne]{Matthias K\"ohne}

\author[J. Saal]{J\"urgen Saal}

\author[L.\ Westermann]{Laura Westermann}

\address{Mathematisches Institut, Angewandte Analysis\\
         Heinrich-Heine-Uni\-ver\-sit\"at D\"usseldorf\\
         40204 D\"usseldorf, Germany}

\email{matthias.koehne@hhu.de}
\email{juergen.saal@hhu.de}
\email{laura.westermann@hhu.de}

\thispagestyle{empty}
\parskip0.5ex plus 0.5ex minus 0.5ex
\bibliographystyle{plain}
\def\spn{{\mathrm{span}}}
\def\sL{{\mathscr{L}}}
\def\sLis{{\mathscr{L}_{is}}}
\def\cL{{\mathcal{L}}}
\def\cT{{\mathcal{T}}}
\def\eps{{\varepsilon}}
\def\R{{\mathbb{R}}}
\def\bt{{\beta}}
\def\pa{{\partial}}
\newcommand{\bN}{\mathbb N}

 



\begin{abstract}
In this note we prove that the solution of the stationary and
the instationary Stokes equations subject to perfect slip boundary
conditions on a 2D wedge domain admits optimal regularity in
the $L^p$-setting, in particular it is $W^{2,p}$ in space.
This improves known results in the literature to a large extend.
For instance, in \cite[Theorem 1.1 and Corollary 3]{Maier} it is 
proved that the Laplace and the Stokes operator in the underlying 
setting have maximal regularity. 
In that result the range of $p$ admitting $W^{2,p}$ regularity, however,
is restricted to the interval $1<p<1+\delta$ for small $\delta>0$, 
depending on the opening angle of the wedge.
This note gives a detailed answer to the question,
whether the optimal Sobolev regularity extends to the
full range $1<p<\infty$. We will show that
for the Laplacian this does only hold on a suitable subspace,
but, depending on the opening angle of the wedge domain, 
not for every $p\in(1,\infty)$ on the entire $L^p$-space. 
On the other hand, for the Stokes operator in the space of
solenoidal fields $L^p_\sigma$ 
we obtain optimal Sobolev regularity for the full range
$1<p<\infty$ and for all opening angles less than $\pi$. 
Roughly speaking, this relies on the fact that
an existing ``bad'' part of $L^p$ for the Laplacian is complemented 
to the space of solenoidal vector fields.
\end{abstract}

\maketitle


\section{Introduction and main results}
It is well-known that regularity properties for PDE 
on non-smooth domains are important
for many applications. 
The main objective of this note is to derive best possible regularity in the
$L^p$-setting for the instationary Stokes equations 
\begin{equation}
	 	\left.
	 	\begin{array}{r@{\ =\ }lll}
	 		\partial_tu - \Delta u + \nabla \pi &  f &
			\text{in} & (0,\infty) \times G, \\
	 		\text{div}\,u &  0 &\text{in} &
			(0,\infty) \times G, \\
	 		\text{curl}\, u=0, \ u \cdot \nu &0 &
			\text{on} & (0,\infty) \times \partial G, \\
	 		u(0) &  u_0 & \text{in} & G,
	 	\end{array}
	 	\right\}
	 \label{stokes_gleichung}
\end{equation}
subject to perfect slip boundary conditions on a 
two-dimensional wedge type domain given as
\begin{equation}
G:= \left\{ (x_1, x_2) \in \mathbb{R}^2: \ 0<x_2<x_1 \tan\theta_0 \right\}.
\label{domain_G}
\end{equation}
Here $ \nu $ denotes the outer normal vector at $\partial G$,
$\theta_0\in(0,\pi)$ the opening angle of the wedge, and 
$\text{curl}\, u=\partial_1 u_2-\partial_2 u_1$. 

There exist approaches to an $L^p$-theory for 
classical elliptic and parabolic problems
on domains with conical boundary points, see, e.g., the classical
monographs \cite{Grisvard,maro2010}, or also \cite{Coster-Nicaise:
Helmholtz equation, Coster-Nicaise: Cauchy-Dirichtlet heat equation} for
the heat equation subject to Dirichlet boudary conditions.
In contrast to that,
corresponding results for the Stokes equations are very rare, in
particular for the instationary case.
For the stationary Stokes equations there are the classical regularity
results \cite{kondra1967,ko1976,dauge1989,maro2010,Grisvard,deuring1998}. 
For a negative result concerning the generation of an 
analytic semigroup in three dimensions for the Stokes operator 
subject to the no-slip condition see \cite{deuring2001}. More recently,
an  approach to analytic regularity was presented in \cite{gs2006}. 
We also refer to \cite{Kozlov-Rossmann}, where the Stokes equations subject to no-slip boundary conditions in a cone are studied, and to \cite{hisa2016} for an overview on the Stokes equations including approaches to non-smooth domains.

It seems that a general  approach to the instationary Stokes equations
on domains with edges and vertices does not exist in the literature, 
even for domains having a simple structure such as wedge domains.
There is, of course, the Lipschitz approach to even more general 
non-smooth domains. Existence and analyticity of the Stokes
semigroup on $L^p_\sigma$ on Lipschitz domains is proved, for instance,
in \cite{mitmon2008,shen2012,tolksdorf2017}. 
Note that the Lipschitz approach does not provide
full $W^{2,p}$ Sobolev regularity which, however, 
might be crucial for the treatment of related quasilinear problems.
Moreover, in the Lipschitz approach the 
range of available $p$ is restricted in general.
Thus, for our purposes this approach seems to be too general. 
The main objective of this note is $W^{2,p}$ Sobolev regularity
for (\ref{stokes_gleichung}) for all $p\in(1,\infty)$. 
\begin{remark}
In this note we frequently use the expressions {\em maximal regularity,
optimal (Sobolev) regularity, optimal regularity in the $L^p$-setting}.
The {\em maximal regularity} here refers to the standard
maximal regularity concerning (linear) evolution equations,
see \cite{Kalton-Weis:Operator-Sums,Denk-Hieber-Pruess:Maximal-Regularity,
Kunstmann, Pruss_Simonett} for a precise definition.
The latter two notions refer to
optimal regularity in the sense of elliptic regularity, i.e.,
for the second order equations considered here to $W^{2,p}$-regularity
in space.
\end{remark}
Concerning Stokes the advantage of imposing perfect slip conditions
lies in the fact that Helmholtz projector and Laplacian commute,
which is not the case in general. Hence the Stokes operator is given
as the part of the Laplacian in the solenoidal subspace.
Note that this observation has been utilized in \cite{mitmon2008} 
and \cite{Maier} already.
In fact, in \cite{Maier} maximal regularity for  
(\ref{stokes_gleichung}) is proved in two and three dimensional
wedges in Kondrat'ev spaces
\begin{equation}\label{defkondsp}
	L^p_{\gamma}(G,\R^2)
	:=L^p(G,\rho^\gamma d(x_1,x_2),\R^2),
	\quad \rho:=|(x_1,x_2)|,\ \gamma\in\R.
\end{equation}
(Note that \cite{Maier} focuses on the 3D version; the 2D counterpart
then is completely analogous.)
Optimal Sobolev regularity in the sense of our main results below, 
however, could only be established for 
$1<p<1+\delta$ with $\delta>0$ possibly small, depending
on the opening angle $\theta_0$ of the wedge and the Kondrat'ev exponent
$\gamma$. This shortcoming relies on 
a spectral constraint that relates to the constraint (\ref{spec_cond})
in Theorem~\ref{main1} below. In fact, for $\gamma=0$ under the constraint
imposed in \cite{Maier}
we even have $\delta\to 0$ for $\theta_0\to \pi$ such that for 
angles close to $\pi$ only a very small interval for $p$ remains.
Note that a similar spectral constraint concerning 
regularity for the heat equation subject to Dirichlet and Neumann
boundary conditions is imposed in \cite{Nazarov2001,Pruss}. 

In this note we will show that in 2D this vast restriction on $p$ can
be dropped completely. To be precise, our main result reads as follows
(see (\ref{defsolss}) and (\ref{defsolss2}) 
for the definition of the solenoidal subspace
$L^p_\sigma(G)$ on a wedge domain and the definition
of $\widehat{W}^{1,p}(G)$).
\begin{theorem}\label{main2}
	Let $1<p<\infty$, $\theta_0\in(0,\pi)$, 
	$\rho=|(x_1,x_2)|$, and $G\subset\R^2$
	be defined as in (\ref{domain_G}). 
	Then the Stokes operator subject to perfect slip 
		\begin{align*}
			A_S u &= - \Delta  u, \\
			u \in D(A_S)   &= \biggl\{ u \in W^{2,p}(G,\R^2)
			\cap L^p_\sigma(G):
			\ \mathrm{curl}\, u =0, \ \nu \cdot u = 0 \ 
			\text{on} \ \partial G, \biggr. 
		\\ & \biggl. 
			 \qquad  \rho^{|\alpha|-2} \partial^\alpha u 
			\in L^p(G, \mathbb{R}^2) \ (|\alpha| \leq 
			2) \biggr\}
		\end{align*}
	is $\cR$-sectorial with $\cR$-angle
	$\phi^\cR_{A_S}<\pi/2$, 
	hence has maximal regularity on $L^p_\sigma(G)$.
\end{theorem}
As an immediate consequence we obtain strong solvability of
(\ref{stokes_gleichung}).
\begin{corollary}\label{main2cor}
	Let $1<p,q<\infty$, $\theta_0\in(0,\pi)$, 
	$\rho=|(x_1,x_2)|$, and $G\subset\R^2$
	be defined as in (\ref{domain_G}). Then for every 
	$f\in L^q\left((0,\infty),L^p_\sigma(G)\right)$ 
	and $u_0\in \cI_{p,q}:=(L^p_\sigma(G),D(A_S))_{1-1/p,q}$ 
	there is a unique solution 
	$(u,\pi)\in \bigl(W^{1,q}(\R_+,L^p_\sigma(G))\cap 
	L^q(\R_+,D(A_S))\bigr)\times L^q(\R_+,\widehat{W}^{1,p}(G))$
	of (\ref{stokes_gleichung}). This solution satisfies $\pi=0$
	and
\[
	\|\partial_t u\|_{L^q(\R_+,L^p)}
	+\sum_{|\alpha|\le 2}\|\rho^{|\alpha|-2} \partial^\alpha u
	\|_{L^q(\R_+,L^p))}
	\le C\left(\|f\|_{L^q(\R_+,L^p)}+\|u_0\|_{\cI_{p,q}}\right)
\]
with $C>0$ independent of $f$ and $u_0$.
\end{corollary}
For the proof of Theorem~\ref{main2} we basically follow the 
strategy in \cite{Maier}, 
that is, we first consider the Laplace equation
subject to perfect slip conditions. In a standard procedure, 
by employing polar coordinates
and Euler transformation, we reduce the Laplace equation on a wedge 
to a problem on a layer. 
On the layer we apply the operator sum method as it is performed
originally in \cite{Pruss}.

The difference to \cite{Maier} lies in the fact
that here we consider the elliptic problem 
\begin{equation}
	 	\left.
	 	\begin{array}{r@{\ =\ }lll}
	 		- \Delta u &  f &
			\text{in} & G, \\
	 		\text{curl}\, u=0, \ u \cdot \nu &0 &
			\text{on} & \partial G 
	 	\end{array}
	 	\right\}
	 \label{lapequ}
\end{equation}
instead of the corresponding resolvent problem. 
The advantage is that for the transformed problem we then 
have precise knowlege
on the spectrum. This, in turn, allows to completely characterize the set
of $p$ for which optimal regularity for (\ref{lapequ}) is available. 
We formulate this in our second main result which also represents the basis
for Theorem~\ref{main2} and which we even prove in Kondrat'ev spaces.
\begin{theorem} \label{main1}
Let $1<p<\infty$, $\theta_0\in(0,\pi)$, 
$\gamma\in\R$, and $\rho=|(x_1,x_2)|$. 
Then equation \eqref{lapequ} is for each $f \in L^p_\gamma(G, \bR^2)$  
uniquely solvable with a solution $u$ satisfying
\begin{equation}\label{estoptellreg}
	 \rho^{|\alpha|-2} \partial^\alpha u\in L^p_\gamma(G, \mathbb{R}^2)
	\quad (|\alpha|\le 2)
\end{equation}
if and only if
\begin{equation}\label{spec_cond}
	2-\frac{2+\gamma}{p}
	\not\in\left \{  \frac{k\pi}{\theta_0} \pm 1: 
	k\in \mathbb{N} \right \} \cup \{1\}.
\end{equation}
\end{theorem} 
\begin{remark}
(a)\ For $\gamma=0$ condition (\ref{spec_cond}) reduces to
\begin{equation}\label{spec_cond_g0}
	2-\frac{2}{p}
	\not\in\left\{1,\, \frac{\pi}{\theta_0}- 1,\,
	\frac{2\pi}{\theta_0}- 1\right\},
\end{equation}
see Subsection~\ref{secconsis}. From this we see that 
for each angle $\theta_0\in(0,\pi)$ the case $p=2$ is excluded.
On the other hand, from the results obtained in \cite{Grisvard} 
one would expect $\partial^\alpha u\in L^2(G,\R^2)$
for $|\alpha|=2$.
Taking into account Hardy's inequality, by which the lower oder terms
in (\ref{estoptellreg}) can be estimated by the second order terms,
this looks curious at a first glance. However, $p=2$ is exactly
the case when Hardy's inequality is not valid.
Thus, for $p=2$
(\ref{estoptellreg}) still can fail for one of the lower order terms,
although $\partial^\alpha u\in L^2(G,\R^2)$, $|\alpha|=2$, might be
true.
For the excluded $p\neq 2$ (\ref{estoptellreg}) must fail for at least
one of the second order terms, since otherwise Hardy's inequality
would yield (\ref{estoptellreg}) to be valid for all terms,
see also Remark~\ref{remconsnotw2}(b). We think that this 
is an effect of the unboundedness of a wedge domain, see
also the lines subsequent to Theorem~\ref{korollarmainstokes}.\\[1mm]
(b)\ Another curious looking case is given by $\gamma=0$ and
$\theta_0=\pi/2$. Then, by reflection arguments the problem on the 
wedge $G$ can be reduced to $-\Delta u=f$ on $\R^2$. This fact implies
$\partial^\alpha u\in L^p(G,\R^2)$, $|\alpha|=2$, to be valid
for all $p\in(1,\infty)$. Again this does not contradict 
the assertion of Theorem~\ref{main1}, since in this case 
(\ref{spec_cond_g0}) is reduced to $2-2/p\not\in\{1\}$.
Thus, only $p=2$ is excluded and we find ourselves in the situation
explained in (a).
\end{remark}

It seems that Theorem~\ref{main1} is not contained
in the previous literature. This might rely on the fact that
due to the boundary conditions (\ref{lapequ}) is a system, whereas in
previous literature the Laplace equation is preferably 
considered as a scalar equation. 

In contrast to Theorem~\ref{main2}, as a first consequence 
of Theorem~\ref{main1} we obtain that for 
the instationary diffusion equation subject to perfect slip
$W^{2,p}$ regularity is not 
available if condition (\ref{spec_cond}) is not fulfilled,
see Theorem~\ref{thmmasneg} below. 
The point why we nevertheless can prove Theorem~\ref{main2} relies
on the fact that the part of $L^p$ destroying $W^{2,p}$ regularity
is more or less complemented to the space
of solenoidal fields $L^p_\sigma(G)$. 
By this fact we obtain 
optimal regularity for the stationary Stokes equations, too.

\begin{theorem} \label{korollarmainstokes}
Let $1<p<\infty$ and $\theta_0\in(0,\pi)$. 
Then for each $f \in L^p_\sigma(G)$ there exists a unique 
solution $(u,\pi)\in \bigl(W^{2,p}(G,\R^2)\cap L^p_\sigma(G)\bigr)
\times \widehat{W}^{1,p}(G)$ of the stationary 
version of (\ref{stokes_gleichung}). This solution satisfies $\pi=0$
and 
\[
	 \rho^{|\alpha|-2} \partial^\alpha u\in L^p(G, \mathbb{R}^2)
	\quad (|\alpha|\le 2).
\]
\end{theorem} 

Of course, the Stokes equations subject to perfect slip in 2D
can also be considered without taking the path via the Laplace
equation, by utilizing its equivalence to a biharmonic equation.   
The authors of this note, however, also wanted
to compare the two equations concerning regularity. 
In this regard, we find it 
most interesting that in the underlying situation the outcome for
the Stokes equations is better than for the Laplace or diffusion
equation, which usually is vice versa by the fact that the Laplacian
enjoys much nicer properties than the Stokes operator.

We also want to remark that the appearance of the weight function 
$\rho$ is essentially due to the unboundedness of a wedge domain.
Since the second order derivatives $\nabla^2 u$ are
$L^p$-functions on the wedge $G$,
the same integrability is valid for $u$ and $\nabla u$ on any bounded
neighbourhood of the vertex of $G$. This follows by standard results,
see, e.g., \cite[Lemma~II.6.1 and Remark II.6.1]{Galdi2012}.
Thus, especially the case $\gamma = 0$ can also be regarded as 
one reply to the question for $W^{2, q}$-regularity for 
Laplace and Stokes equations subject to perfect slip on 
a convex domain.

We outline the strategy of the proofs and the organization of this
note. 
Section~\ref{seclapps} contains the approach to the Laplace operator
and equation.
After fixing notation and transforming from a wedge to a layer,
in Subsection~\ref{sec_prob_transformiert} 
we establish optimal regularity for the transformed problem.
This is based on operator sum methods, that is, Kalton-Weis type
theorems.
Since the transform from a wedge to a layer is a diffeomorphism,
this gives instantly Theorem~\ref{main1}, as stated in
Subsection~\ref{secoptkeil}.
To carry over regularity from the elliptic problem 
(\ref{lapequ}) to the instationary diffusion equation, it is enough 
to show optimal regularity for the resolvent problem
\begin{equation}
	 	\left.
	 	\begin{array}{r@{\ =\ }lll}
	 		(1- \Delta) u &  h &
			\text{in} & G, \\
	 		\text{curl}\, u=0, \ u \cdot \nu &0 &
			\text{on} & \partial G. 
	 	\end{array}
	 	\right\}
	 \label{laprespb}
\end{equation}
The idea is to regard $u$ as the solution of the elliptic
problem (\ref{lapequ}) with right-hand side $f=h-u\in L^p(G,\R^2)$. 
According to Theorem~\ref{main1} we know that this problem
has a solution, say $v$, with the regularity given in 
(\ref{estoptellreg}). It remains to prove $u=v$. By the outcome
given in \cite{Maier} this is valid for $p>1$ close to $1$.
This means, if the solution operators to problems 
(\ref{lapequ}) and (\ref{laprespb}) are consistent on the 
scale $(L^p(G,\R^2))_{1<p<\infty}$, the 
regularity in (\ref{estoptellreg}) transfers to the solution $u$
of (\ref{laprespb}) for all $1<p<\infty$. 
By the equivalence in Theorem~\ref{main1}, however, consistency
for the solution operator of (\ref{lapequ}) cannot hold on
the full scale $(L^p(G,\R^2))_{1<p<\infty}$. 
But, as shown in Subsection~\ref{secconsis}, 
it is consistent on a suitable scale of ``nice'' subspaces.   
This leads in Subsection~\ref{secreglap} to optimal regularity
for the diffusion equation on the subspaces 
for all $1<p<\infty$ (see Theorem~\ref{mainlap}).

A major difficulty for the transference of optimal Sobolev regularity 
to the Stokes equations is given by the fact that
the space of solenoidal fields $L^p_\sigma(G,\R^2)$ is not 
directly included in the ``nice'' subspace of $L^p$.
A crucial issue, taking the major part of Section~\ref{secstokes},
is therefore to prove that it can be 
isomorphically embedded into this subspace.
This isomorphic embedding is also valid for the domains of
the involved operators,
finally leading to 
Theorem~\ref{main2} and Theorem~\ref{korollarmainstokes}.

\section{The Laplace operator on a wedge domain subject to perfect slip}
\label{seclapps}

\subsection{Notation} 
First we introduce the notation used throughout this note. Let $X$ be a
Banach space. For $1 \leq p \leq \infty$ and a measure space $(S,
\Sigma, \mu)$, we denote by $L^p(S, \mu, X)$ the usual Bochner-Lebesgue
space. If $1 \leq p \leq \infty$ and $(S, \Sigma, \mu)$ is a complete
measure space, then $L^p(S, \mu, X)$ is a Banach space.
If $\Omega \subset \bR^n$ is a domain and $\mu$ is the (Borel-) Lebesgue
measure, we write $L^p(\Omega, X)$. We define the Sobolev space of order
$k \in \bN_0$ as $W^{k,p}(\Omega,\R^n)$, where $W^{0,p}:=L^p$.

Let $G\subset\R^2$ be the wedge domain defined in \eqref{domain_G} and 
let $\rho=\rho(x_1,x_2) = |(x_1,x_2)|$. 
We set
$$K^m_{p,\gamma}(G, \mathbb{R}^2):= \{ u \in
L^1_{loc}(G, \mathbb{R}^2): \ \rho^{|\alpha|-m} \partial^\alpha u 
\in L^p_\gamma(G, \mathbb{R}^2), \ |\alpha| \leq m \}$$
where $\alpha\in\N^m$ denotes a multiindex, $\gamma\in\R$,
and $L^p_\gamma(G, \mathbb{R}^2)$ is defined as in (\ref{defkondsp}). 
Then $K^m_{p,\gamma}(G,\R^2)$ equipped with
$$ \|u \|_{K^m_{p,\gamma}}:=\|u \|_{K^m_{p,\gamma}(G,\R^2)}:= \biggl(\sum_{|\alpha| \leq m} \|
\rho^{|\alpha|-m} \partial^\alpha u \|_{L^p_\gamma(G,\R^2)}^p\biggr)^{1/p}$$
is a Banach space. 
We also set $K^m_p(G,\R^2):=K^m_{p,0}(G,\R^2)$. 
Let $1<p<\infty$ with $1/p + 1/p'=1$. If $u\in L^p(\Om ,\bR^2)$ and $v
\in L^{p'}(\Om, \bR^2) $ we denote the duality pairing by $(u, v):= (u, v)_\Om := \int_\Om u v dx$.
For a family $(x_j)_{j\ge 1}$ of elements in a linear space $X$, we denote by
\[
	\la x_j\ra_{j\ge 1}=\la x_1,x_2,\ldots\ra
\]
its linear hull. 

For Banach spaces $X$, $Y$ the space of bounded linear 
operators from $X$ to $Y$ is denoted by $\mathscr{L}(X,Y)$, where
$\mathscr{L}(X):=\mathscr{L}(X,X)$. The subclass of isomorphisms is
denoted by $\mathscr{L}_{is}(X,Y)$ or $\mathscr{L}_{is}(X)$,
respectively. 
If $X'$ is the dual space of $X$, then we use for the corresponding
duality pairing the notation 
\[
	\la x',\, x\ra_{X',X},\quad x\in X,\ x'\in X'.
\]
We denote for a linear operator $A$ in $X$ domain and
range by $D(A)$ and $R(A)$. Its spectrum, point spectrum, 
and resolvent set are written 
as $\sigma(A)$, $\sigma_p(A)$, and $\rho(A)$.
We say that an operator $A:D(A)\subset X\to X$ is sectorial, if
$\overline{D(A)}=\overline{R(A)}=X$, $(0,\infty)\subset\rho(-A)$, 
and the family $(\lambda(\lambda+A)^{-1})_{\lambda>0}$ is uniformly bounded.
If the latter family is $\cR$-bounded, then we call $A$
$\cR$-sectorial. By $\phi_A$ and $\phi^{\cR}_A$ we denote 
the corresponding spectral and $\cR$-angle, respectively 
\cite{Kalton-Weis:Operator-Sums,Denk-Hieber-Pruess:Maximal-Regularity,
Kunstmann}.

In this note we also employ elements of the $\cH^\infty$-calculus
(e.g.\ in Theorem~\ref{theorem_haupt}). 
By $\mathcal{H}^\infty(X)$ we denote the class of all operators
$A$ in $X$ admitting a bounded $\mathcal{H}^\infty$-calculus on $X$.
The corresponding $\mathcal{H}^\infty$-angle is denoted by 
$\phi^\infty_A$. We refer to
\cite{Kalton-Weis:Operator-Sums,Denk-Hieber-Pruess:Maximal-Regularity,
Kunstmann, Pruss_Simonett} for an introduction into $\cH^\infty$-calculus,
$\cR$-boundedness, and related notions.

\subsection{Transformation of the elliptic linear problem}\label{first}
In this section we transform the elliptic linear problem
\eqref{lapequ} on a two-dimensional wedge domain onto a layer domain 
of the form $\Omega:=\bR \times I$. 
If $\theta_0$ denotes the angle of the wedge
$G$ we set $I:=(0, \theta_0)$. In the first step we introduce polar coordinates
whereas in the second step we employ the Euler transformation. Last we 
rescale the appearing terms such that we can work in the transformed setting 
in unweighted $L^p$-spaces. 

We write the inverse of the transform to polar coordinates as 
	$$ \psi_P: \mathbb{R}_+ \times I \rightarrow G, \ \ \ (r, \theta) \mapsto (r \cos \theta, 
	r \sin \theta) =(x_1, x_2) $$ 	
with the associated orthogonal basis 
	$$ e_r= \begin{pmatrix} \cos \theta \\ \sin \theta \end{pmatrix}, \ \ \ \ \ e_\theta = 
	\begin{pmatrix} -\sin \theta \\ \cos \theta \end{pmatrix}. $$
We identify the orthogonal transformation matrix $\cO$ of the components of a vector field as 
\begin{equation*}
	\mathcal{O} = \begin{pmatrix} \cos \theta & - \sin \theta \\ \sin \theta & \cos \theta 
	\end{pmatrix} .	
\end{equation*}
Next, we employ Euler transformation $r=e^x$ in radial direction, where
by an abuse of notation we write $x \in \bR$ for the new variable. We set
	$$ \psi_E : \Omega \rightarrow \mathbb{R}_+ \times I, \ \ \ (x, \theta) \mapsto (e^x, 
	\theta) =: (r, \theta). $$
It is not difficult to see that 
	$$\psi:= \psi_P \circ \psi_E: \Omega \rightarrow G$$
is a diffeomorphism. We set 
\[
	\Psi u:= u\circ\psi \quad\text{and}\quad
	\Psi^{-1} v:= v\circ \psi^{-1}.
\]
For $\alpha\in\R$ we also denote the multiplication operator by
\[
	M_\alpha v:=e^{\alpha x} v.
\]
Analogous to \cite{Maier} we define 
pull back resp.\ push forward by
\begin{equation}
	v:= \Theta^{*}_pu:= M_{-\beta_p} \mathcal{O}^{-1} \Psi u  \ \ \
	\text{resp.} \ \ u= \Theta_*^pv=
	\Psi^{-1}\mathcal{O}M_{\beta_p} v 
\label{pullback}
\end{equation}
with $\beta_p \in \bR$ to be chosen later.
Then the transformed Laplacian, computed straight
forwardly, is given as
\begin{equation*}
			\Theta^*_p(\Delta u)=   e^{- 2x} \begin{pmatrix}
			r_p(\partial_x) v_x +
			\partial^2_\theta v_x - 
			v_x - 2 \partial_\theta v_\theta \\
			r_p(\partial_x) v_\theta  +\partial^2_\theta v_\theta  - v_\theta + 2\partial_\theta 
			v_x \end{pmatrix}
\label{nummer_transproblem}
\end{equation*}
with the polynomial	
\begin{equation}\label{defpoly}
r_p(\partial_x):= \partial_x^2+ 2 \beta_p \partial_x + \beta_p^2.
\end{equation}
To absorb the factor $e^{-2x}$, we put 
\begin{equation}
	g=(g_x, g_\theta):= \widetilde{\Theta}^*_pf := e^{2x} \Theta^*_p f
\label{substitution}
\end{equation}
so that
\begin{equation*}
	\int_{\mathbb{R}} | g(x, \theta) |^p dx  =  \int_0^\infty |
	r^{2-\beta_p} \mathcal{O}^{-1} f(\psi_p(r, \theta)) |^p
	\frac{dr}{r}.
\end{equation*}
Then by the choice $p(2-\beta_p)= \gamma + 2$, that is 
\begin{equation}
		\beta_p =2 - \frac{2+\gamma}{p},
	\label{bedingung}
\end{equation}
we see
that in the transformed setting we can work in an unweighted $L^p$-space, 
see \cite{Pruss,Maier}. Notice that by this choice of $\beta_p$
also pull back and push forward depend on $p$, i.e., the corresponding
families are not consistent in $p$.

Finally, we transform the boundary conditions $ \nu \cdot u= 0, \ \text{rot} \ u =0 \ \text{on} \  \partial G$ of the problem \eqref{lapequ} 
to the result that
	$$ \partial_\theta v_x =0,  \ \ \  v_\theta=0 \ \ \ \text{on} \ \ \partial \Omega =\bR \times \{0, \theta_0 \}.  $$
Summarizing,  we receive the following transformed problem
on $\Omega=\R\times I$:
\begin{equation}
	\left.
	\begin{aligned}
		r_p(\partial_x) v_x + \partial^2_\theta v_x - v_x - 2 \partial_\theta v_\theta &= g_x \ \ \text{in} \ \Omega, \\
		r_p(\partial_x) v_\theta +\partial^2_\theta v_\theta - v_\theta + 2 \partial_\theta v_x &= g_\theta \ \ \ \text{in} \ \Omega, \\
		\partial_\theta v_x=0, \ \ v_\theta &=0 \ \ \ \text{on} \ \partial \Omega.		
	\end{aligned}
	\right\}
	\label{prob_transformiert}
\end{equation}

\subsection{Optimal elliptic regularity for the transformed problem}
\label{sec_prob_transformiert}

Here we consider problem (\ref{prob_transformiert}).
In this subsection we frequently identify $L^p(\Omega,\bR^2)$
with its isometrically isomorphic version
$L^p\bigl(\bR,L^p(I,\bR^2)\bigr)$, often without further notice.
We introduce the operators associated to the single 
parts in \eqref{prob_transformiert}:
\\[2mm] 
(1)\ Let $r_p$ be the polynomial given in (\ref{defpoly}). 
We define $\cT_{p,x}$ in $L^p(\mathbb{R})$ by means of
		$$\cT_{p,x}u := r_p(\partial_x ) u, 
		\quad u \in D(\cT_{p,x}):=W^{2,p}(\mathbb{R}).$$ 
The spectrum of $\cT_{p,x}$ is given by the 
parabola $r_p(i\mathbb{R})$ which is
symmetric about the real axis, open to the left and has its intersection
point with the $x$-axis at $\beta_p^2$ with $\beta_p$ as in \eqref{bedingung}. 
It is straight forward to show that $-\cT_{p,x}+b \in \mathcal{H}^\infty
(L^p(\mathbb{R}))$ for $b > \beta_p^2$ with $ \phi_{-\cT_{p,x}+b}^\infty <
\pi/2$, e.g., by the use of Fourier transform, see also
\cite{Pruss,Maier}. By means of operator-valued Fourier multiplier
results \cite{Weis,Denk-Hieber-Pruess:Maximal-Regularity,Kunstmann} 
these facts obviously transfer to the vector-valued
version on $L^p(\bR,L^p(I,\bR^2))$ given as
\[
	T_{p,x}u := \cT_{p,x} u, \quad u \in D(T_x):=W^{2,p}(\bR,L^p(I,\bR^2)).
\]
(2)\ We define $\cT_{p,\theta}$ in $L^p(I, \mathbb{R}^2)$ by
				$$\cT_{p,\theta}v:= \begin{pmatrix} \partial_{\theta}^2-1  & -2\partial_{\theta} \\  2 \partial_{\theta}    &  \partial_\theta^2  -1 \end{pmatrix} v  $$		
			on $D(\cT_{p,\theta}):=\{v =(v_x, v_\theta) \in
			W^{2,p}(I, \mathbb{R}^2):
			\partial_{\theta}v_x=0,\ v_\theta=0 \ \text{on}\
			\partial I \}$. 
It is also straight forward to identify
\begin{equation}\label{spectt} 
	\sigma(\cT_{p,\theta})=\sigma_p(\cT_{p,\theta})
	=  \left \{ -\left( \frac{k\pi}{\theta_0} \pm 1 \right )^2: k\in \mathbb{N} \right \} \cup \{-1\} 
\end{equation}
as its spectrum with corresponding eigenfunctions
$(v_x^k,v_\theta^k)^\tau$, where
$$v_x^k(\theta):=\cos \left (\frac{k\pi}{\theta_0}\theta \right ), \ \ \
v_\theta^k(\theta):= \pm \sin \left(\frac{k\pi}{\theta_0} \theta \right
),\ \ k \in \mathbb{N}_0 , \ \theta \in I,$$
see also \cite{Maier}. 
Note that $\cT_{p,\theta} $ is self-adjoint in $L^2(I,\bR^2)$. 
Hence the eigenfunctions represent a basis of $L^2(I, \bR^2)$. 
We denote by $(\lambda_i)_{i\in\bN_0}$ the set of eigenvalues,
i.e., $(\lambda_i)_{i\in\bN_0}=\sigma(\cT_{p,\theta})$ such that
$\lambda_0=-1$ and $\lambda_1>\lambda_2>\ldots$.
Setting $e_0:=\left(1/\sqrt{\theta_0},0\right)^\tau$ 
and $e_i:= \frac{\widetilde{e_i}}{\sqrt{\theta_0}}$ for $i\in\N$ 
where $\widetilde e_i$ 
denotes the eigenfunction to the eigenvalue $\lambda_i$, we have
$$ ( e_i, e_j) = \frac{1}{\theta_0} \int_0^{\theta_0}
\widetilde{e_i}\cdot \widetilde{e_j} \ d \theta
=\delta_{ij}.$$
By Fourier series techniques it is also standard to prove that
$-\cT_{p,\theta}$ admits an  $\mathcal{H}^\infty$-calculus on $ L^p(I,
\mathbb{R}^2) $ with $\phi_{-\cT_{p,\theta}}^\infty =0$. The
same properties remain valid for the canonical extension to
$L^p(\bR,L^p(I,\bR^2))$ denoted by
\[
	T_{p,\theta} v:=\cT_{p,\theta} v,\quad
	D(T_{p,\theta}):=L^p(\bR,D(\cT_{p,\theta})).
\]

Optimal regularity for (\ref{prob_transformiert}) is then 
reduced to invertibility of the operator
\begin{equation}\label{sumlxlt}
	T_p:=T_{p,x}+T_{p,\theta}:D(T_{p,x})\cap D(T_{p,\theta})\to L^p(\Omega,\bR^2),
\end{equation}
if we can also show that
\begin{equation}\label{domsumlxlt}
\begin{split}
D(T_{p})&:= \left\{ v= (v_x, v_\theta) \in W^{2,p}(\mathbb{R}
\times I, \mathbb{R}^2), \ \partial_\theta v_x =v_\theta=0 \ \text{on} \
\partial \Omega  \right\}\\
&=D(T_{p,x})\cap D(T_{p,\theta}).
\end{split}
\end{equation}

The proof of these facts requires some preparation. 
Let 
\begin{equation}\label{defpmproj}
	P_{m,p}^cu:= \Sigma_{i=1}^m (u, e_i)e_i
\end{equation}
be the projection of $u \in L^p(I, \bR^2)$
to $ \la e_1, ..., e_m\ra$. 
We set $P_{m,p}:=1-P_{m,p}^c$ and 
$E^p_m:=P_{m,p}\left(L^p(I,\bR^2)\right)$, i.e., $E^p_m$ is the
complement to $\la e_1,...,e_m\ra$. Note that $(P_{m,p})_{1<p<\infty}$
is a consistent family. By this fact we omit the index $p$ and 
write just $P_m$.
We denote the extension of $P_{m}$ to $L^p(\bR,L^p(I,\bR^2))$ by $\mP_{m}$.
Obviously then $\mathbb{P}_{m}\in \mathscr{L}(L^p(\Omega, \mathbb{R}^2))$ 
is a projector as well and we have
\begin{equation}
\label{lemm_projektion}
	L^p(\Omega, \mathbb{R}^2) = L^p(\mathbb{R},
	\la e_1,...,e_m\ra) \oplus L^p(\mathbb{R},E^p_m).
\end{equation}
The following properties are obvious.
\begin{lemma}\label{lemm_kommutativ}
	Let $T_{p,x}$ and $T_{p,\theta}$ in $L^p(\Omega, \bR^2)$ for
	$1<p<\infty$ be defined as above and let $b>\beta_p^2$
	with $\beta_p$ as given in (\ref{bedingung}).  
	Then we have
		\begin{enumerate}
			\item $\mathbb{P}_{m}u\in D(T_{p,i})$ and
			$\mathbb{P}_{m} T_{p,i} u= T_{p,i}\mathbb{P}_{m}u$
			for $u\in D( T_{p,i})$ and $i\in \{\theta, x\}$;
			\item $-T_{p,x} + b, -T_{p,\theta} \in \mathcal{H}^\infty(L^p(\mathbb{R},E^p_m)) \cap 
			\mathcal{H}^\infty (L^p(\mathbb{R}, \la
			e_1,...,e_m\ra ))$ with the 
			corresponding angles $\phi_{-T_{p,x}+b}^\infty < \frac{\pi}{2}$ and $\phi_{-T_{p,\theta}}^\infty=0$;
			\item $\mathbb{P}_{m} R(\lambda, T_{p,i}) = R(\lambda, T_{p,i}) \mathbb{P}_{m}$ for 
			$\lambda \in \rho (T_{p,i})$ and $i\in\{\theta, x\}$;
			\item $(\lambda- T_{p,x})^{-1} 
			(\mu- T_{p,\theta})^{-1}
			= (\mu- T_{p,\theta})^{-1} 
			(\lambda- T_{p,x})^{-1}$ for 
			$\lambda \in \rho( T_{p,x} )$ and $\mu \in \rho(T_{p,\theta}).$
\end{enumerate}
\end{lemma}

The domains of the Operators $T_{p,x}$ and $T_{p,\theta}$
in the subspace $L^p(\R,E^p_m)$ are defined as
\begin{equation}\label{domstm}
\begin{split}
	D_m(T_{p,x})&:=D(T_{p,x})\cap L^p(\R,E^p_m)
	\quad\text{and}\\
	D_m(T_{p,\theta})&:=D(T_{p,\theta})\cap L^p(\R,E^p_m)
\end{split}
\end{equation}
respectively. The assertions of Lemma~\ref{lemm_kommutativ} then
easily yield
\begin{corollary}\label{lemprojdoms}
The operator $\mathbb P_{m}$ is a projector on $D(T_{p,i})$
and we have
\begin{enumerate}
\item $D_m(T_{p,i})=\mathbb P_{m}\left(D(T_{p,i})\right)$,
\item $D(T_{p,i})=D_m(T_{p,i})\oplus (1- \mathbb P_{m})D(T_{p,i})$
\end{enumerate}
for $i\in\{\theta,x\}$. 
\end{corollary}

We will characterize the invertibility of the operator in (\ref{sumlxlt}) by 
employing the operator sum method. More precisely, we apply
\cite[Proposition~3.5]{nauasaal} which is obtained as a consequence of
the Kalton-Weis theorem \cite[Corollary~5.4]{Kalton-Weis:Operator-Sums}.

\begin{theorem} \label{theorem_haupt}
	Let $1<p< \infty$ and $\beta_p = 2- \frac{2+\gamma}{p}$.
	Then
			$$T_{p,\theta}+T_{p,x} \in
			\mathscr{L}_{is}\left(D(T_{p,\theta})\cap D(T_{p,x}),
			L^p(\Omega, \mathbb{R}^2)\right)$$	
	if and only if $-\beta_p^2\not\in  
	\sigma(T_{p,\theta})$. 
\end{theorem}
\begin{proof} 
Assume that $-\beta_p^2\not\in \sigma(T_{p,\theta})$
and that $b>\beta_p^2$.
The fact that $-\beta_p^2\not\in  \sigma(T_{p,\theta})$ guarantees  
\begin{equation}\label{isececond}
\sigma(-T_{p,x}) \cap \sigma(T_{p,\theta}) =\emptyset.
\end{equation}
We first show that $-T_{p,\theta}-T_{p,x}-\varepsilon \in \mathcal{H}^\infty
(L^p(\mathbb{R},E^p_m))$ for some $\varepsilon >0$, which 
essentially gives the assertion. 

To this end, pick $m \in \mathbb{N}_0$ so that $- \lambda_{m+1}>b$ with
$\lambda_{m+1} \in \sigma(T_{p,\theta})$. This implies 
$\sigma(-T_{p,\theta}) \subset (b, \infty)$ on  $L^p(\mathbb{R}, E^p_m)$ 
and hence $0 \in \rho(-T_{p,\theta}-b-\eps)$ for some $\eps>0$.
This fact, Lemma~\ref{lemm_kommutativ}(2) and a 
standard perturbation argument for $\cH^\infty$-calculus \cite[Corollary
5.5.5]{Haase} yield that the shifted operator $-T_{p,\theta}-b-\varepsilon$ 
still satisfies 
	$$-T_{p,\theta}-b-\varepsilon \in \mathcal{H}^\infty(L^p(\mathbb{R}, E^p
	_m)) \ \text{with}\ 
	\phi_{-T_{p,\theta}-b-\varepsilon}^\infty=0.$$
Thanks to Lemma~\ref{lemm_kommutativ}(2), which yields 
$\phi_{-T_{p,\theta}-b-\varepsilon}^\infty + \phi_{-T_{p,x}+b}^\infty <
\pi$, and to Lemma~\ref{lemm_kommutativ}(4) we may apply
\cite[Corollary~5.4]{Kalton-Weis:Operator-Sums} (see also
\cite[Proposition~3.5]{nauasaal}) to the result that
	$$ -T_{p,\theta}-T_{p,x} - \varepsilon 
	= -T_{p,\theta}-b - \varepsilon + (-T_{p,x}+b) \in \mathcal{H}^
	\infty(L^p(\mathbb{R},E^p_m))$$
with $\phi_{-T_{p,\theta}-T_{p,x}-\varepsilon}^{ \infty} \leq \ \text{max}
\{\phi_{-T_{p,\theta}-b}^\infty, \phi_{-T_{p,x}+b}^\infty\}$. 
Particularly, we obtain $0\in
\rho(-T_{p,\theta}-T_{p,x})$, hence 
\begin{equation}\label{invcem}
T_{p,\theta}+T_{p,x}\in\sLis\bigl(D_m(T_{p,x})\cap D_m(T_{p,\theta}),
\, L^p(\mathbb{R},E^p_m)\bigr). 
\end{equation}

For the invertibility of the operator 
$T_{p,\theta}+T_{p,x}$ on $L^p(\mathbb{R},  \la e_1, ..., e_m\ra)$
observe that due to (\ref{isececond}) we have $\lambda_i\in\rho(-T_{p,x})$
on $L^p(\Omega,\R^2)$
for each $\lambda_i\in\sigma_p(T_{p,\theta})$.
Thus 
\[
	\lambda_i+T_{p,x}:L^p(\mathbb{R}, \la e_1, ..., e_m\ra)
	\cap D(T_{p,x})
	\to L^p(\mathbb{R}, \la e_1, ..., e_m\ra)
\]
is invertible. By the fact that
\[
	(T_{p,x}+T_{p,\theta})^{-1}f=\sum_{i=1}^m (\lambda_i+T_{p,x})^{-1}
	(f, e_i) e_i,\ f\in L^p(\mathbb{R}, \la e_1, ...,
	e_m\ra),
\]
we conclude that 
\begin{equation}\label{inve1em}
T_{p,\theta}+T_{p,x}\in\sLis\bigl(L^p(\mathbb{R}, \la e_1, ..., e_m\ra)
\cap D(T_{p,x}),\, L^p(\mathbb{R}, \la e_1, ..., e_m\ra) \bigr).
\end{equation}
Gathering (\ref{lemm_projektion}), (\ref{invcem}), and (\ref{inve1em})
we end up with 
\[
T_{p,\theta}+T_{p,x}\in\sLis\bigl(D(T_{p,x})\cap D(T_{p,\theta}),\, L^p(\Omega,\R^2)\bigr).
\]

Now, assume that $-\beta_p^2\in \sigma(T_{p,\theta})$.
Then the symbol $\lambda+r_p(i\xi)$ of the operator $T_{p,\theta}+T_{p,x}$ 
vanishes exactly at $(\lambda,\xi)=(-\beta_p^2,0)$, where 
$\lambda\in\sigma(T_{p,\theta})$.
Thus, $(\lambda+r_p(i\cdot))^{-1}$ is not bounded, hence not
an $L^p(\R,L^p(I,\R^2))$-multiplier.
This gives the assertion.
\end{proof}

\begin{remark}\label{rem_ex_p}
An inspection of the proof of Theorem~\ref{theorem_haupt} shows 
that we even have that $-T_{p,x}-T_{p,\theta}-\varepsilon \in
\mathcal{H}^\infty(L^p(\Omega,\R^2))$ with
$\phi^\infty_{-T_{p,x}-T_{p,\theta}-\varepsilon}<\pi/2$ 
for some $\eps>0$.
\end{remark}

To obtain optimal regularity we show (\ref{domsumlxlt}).
\begin{lemma}\label{fullellreg}
Let $1<p< \infty$. Then we have
\begin{equation*}
D(T_p)=D(T_{p,\theta})\cap D(T_{p,x}).
\end{equation*}
\end{lemma}
\begin{proof}
Considering the function $\xi \mapsto \frac{i \xi_i \cdot i
\xi_j}{|\xi|^2}|\xi|^2$ for $\xi \in \bR^2$ and applying Mihklin's
Multiplier Theorem \cite{Stein} it is not difficult to see that
$$ W^{2,p}(\bR^2, \bR^2) =L^p(\bR, W^{2,p}(\bR, \bR^2)) \cap W^{2,p}(\bR, L^p(\bR, \bR^2)) $$
with equivalent norms.
The validity of \eqref{domsumlxlt} is proved via an extension theorem,
i.e., via a bounded operator $E:W^{2,p}(\Omega, \bR^2) \rightarrow
W^{2,p}(\bR^2, \bR^2)$ with $Ef |_\Omega =f$ for all $f \in W^{2,p}
(\Omega, \bR^2)$. See \cite[Theorem 4.26]{adams} for the 
existence of $E$. 
\end{proof}


\subsection{Optimal elliptic regularity for problem (\ref{lapequ})}
\label{secoptkeil}

We next consider equivalence of the problems
\eqref{lapequ} and \eqref{prob_transformiert}. 
The Laplace operator on the wedge domain is defined as
	$$ B_p u := -\Delta u, \  u \in D(B_p)
	:=  \{ u \in K^2_{p,\gamma}(G, \mathbb{R}^2): \ \text{curl} \ u=0, \ \nu \cdot u=0 \ \text{on} \ \partial G \}. $$
Observe that the boundary conditions are defined in a local sense.
Indeed, each $u\in K^2_{p,\gamma}(G, \mathbb{R}^2)$ is locally away from
the vertex $(0,0)$ a $W^{2,p}$-function for which the traces are 
well-defined.

\begin{lemma} \label{lemma_result}
Let $1<p<\infty$. 
Let $\Theta_*^p,\widetilde{\Theta}_*^p,\Theta^*_p,\widetilde{\Theta}^*_p$ be defined as
in Subsection~\ref{first}. 
Then we have
\begin{equation}\label{transformsiso}
\widetilde{\Theta}_*^p \in \mathscr{L}_{is}\left(L^p(\Omega, \bR^2),\, 
L^p_\gamma(G, \bR^2)\right),\qquad 
\Theta_*^p \in \mathscr{L}_{is}\left(D(T_{p}),\, D(B_p)\right)
\end{equation}
where $\|\cdot \|_{D(B_p)} = \| \cdot\|_{K^2_{p,\gamma}(G, \mathbb{R}^2)}$ and  
$\|\cdot \|_{D(T_{p})} = \|\cdot \|_{W^{2,p}(\Omega, \mathbb{R}^2)}$.

In particular, $u \in D(B_p)$ is the unique solution of 
\eqref{lapequ} to the right-hand side $f \in L^p_\gamma(G, \R^2)$ 
if and only if $v=\Theta^*_pu \in D(T_{p})$ is the unique solution of 
\eqref{prob_transformiert} to the right-hand side $g= \widetilde{\Theta}^*_pf$.
\end{lemma}

\begin{proof}
By utilizing the transformations given in Subsection \ref{first} and by the 
definition of $\widetilde{\Theta}_*^p$ and $\Theta_*^p$, it is straight forward 
to verify (\ref{transformsiso}). 
Hence problem \eqref{lapequ} and problem \eqref{prob_transformiert} 
are equivalent. 
\end{proof}

Since $-\beta_p^2\not\in \sigma(T_{p,\theta})$ is precisely
condition (\ref{spec_cond}), Theorem~\ref{theorem_haupt}, 
Lemma~\ref{fullellreg}, and Lemma~\ref{lemma_result} now imply 
our second main result Theorem~\ref{main1}.

\begin{remark}\label{remconsnotw2}
(a)\ Theorem~\ref{main1} in particular 
implies that $(B_p^{-1})_{1<p<\infty}$ cannot be a consistent family
on the scale $\left(L^p(\Omega,\R^2)\right)_{1<p<\infty}$.
Otherwise it would be possible to recover the excluded $p$ subject to
condition (\ref{spec_cond}) by an interpolation argument. 
By the equivalence in Theorem~\ref{main1} this, 
however, is not possible.\\[1mm]
(b)\ 
Note that for $\gamma=0$ we have
\[
	\int_G\left|u(x_1,x_2)/|(x_1,x_2)|^{2}\right|^p\,dx_1 dx_2
	=\int_0^{\theta_0}\int_\R \left|e^{-(2-2/p)x}
	u(\psi(x,\theta))\right|^p\,dxd\theta.
\]
Thus, employing twice Hardy's inequality on the $x$ integral, the terms 
$\rho^{|\alpha|-2}\partial^\alpha u$ for $|\alpha|\le 1$ 
can be estimated by the second order terms. 
This, however, does only work provided $2-|\alpha|-2/p\neq 0$
which means at the end that $p\neq 2$,
since otherwise Hardy's inequality is not applicable.
As a consequence, Theorem~\ref{main1} implies that 
\[
	(\partial_j\partial_k u)_{1\le j,k\le 2}\not\subset L^{p}(G,\R^8),
\]
if condition (\ref{spec_cond}) is not satisfied and $p\neq 2$.
In the case $p=2$ second order derivatives might belong
to $L^{2}(G,\R^2)$, but then at least one of the terms
$\rho^{|\alpha|-2}\partial^\alpha u$, $|\alpha|\le 1$, cannot
be in $L^{2}(G,\R^2)$.
\end{remark}

\subsection{Consistency of $(B_p^{-1})_{1<p<\infty}$ on a subscale}
\label{secconsis}

Observe that condition \eqref{spec_cond} is always fulfilled if every
eigenvalue $\lambda_i$ of $T_{p,\theta}$ 
satisfies
	\begin{equation}\label{eigenvalue}
		 \lambda_i < - \left(2- \frac{2+\gamma}{p} \right)^2.
	\end{equation}
As our main interest 
concerns the Stokes equations in $L^p_\sigma(G)$,
from now on we restrict ourselves to the case $\gamma=0$, i.e.,
to the case of Kondrat'ev weight $\rho^\gamma\equiv 1$. Then we have 
\[
	-\beta_p^2=-\left(2- \frac{2}{p} \right)^2\ge -4
	\quad (1<p<\infty).
\]
From (\ref{spectt}) it is easily seen that 
\[
	\lambda_i<-4 \quad (i\ge 3).
\]
Thus, relation \eqref{eigenvalue} remains true for all 
$\lambda_i \in \sigma(T_{p,\theta})$ with $i\ge 3$. 

As we will see later (Proposition~\ref{closedrangeq}), 
excluding the eigenfunctions $e_0$, $e_1$, $e_2$
to the eigenvalues $\lambda_0$, $\lambda_1$, $\lambda_2$
of the transformed operator $T_{p,\theta} $, 
will play no significant role for the Stokes equations.
Roughly speaking, this is due to the fact that
their linear hull in $L^p(\Omega,\R^2)$ 
does not contain divergence free vector fields.
Hence, from now on we consider 
\[
	L^p(\R,E^p_3)=\mathbb{P}_{3}\left(L^p(\Omega,
	\mathbb{R}^2)\right)
\]
as the base space for $T_p:D_3(T_p)\to L^p(\R,E^p_3)$ with the
projector $\mathbb{P}_{3}$  defined in $\eqref{lemm_projektion}$ and 
domain 
\[
	D_3(T_p):=D(T_p)\cap L^p(\R,E^p_3)
	=D_3(T_{p,\theta})\cap D_3(T_{p,x}),
\]
with $D_3(T_{p,\theta})$ and $D_3(T_{p,x})$ as given in (\ref{domstm}).
As an immediate consequence of Theorem~\ref{theorem_haupt}
(and its proof for $m=3$, in particular (\ref{invcem})) we obtain
\begin{corollary}\label{tpinvss}
	We have $T_p\in
			\mathscr{L}_{is}\left(D_3(T_{p}),\,
			L^p(\R,E^p_3)\right)$ for all $1<p<\infty$.	
\end{corollary}
By Lemma~\ref{lemma_result} $\widetilde\Theta_*^p$ and $\Theta_*^p$ are
isomorphisms with inverse $\widetilde\Theta^*_p$ and $\Theta^*_p$,
respectively. This implies that 
\begin{equation}\label{defprojq}
\begin{split}
	\widetilde{\mathbb Q}_{p}&
	:=\widetilde\Theta_*^p\mathbb P_{3}\widetilde\Theta^*_p\quad\text{and}\\
	\mathbb Q_{p}&:=\Theta_*^p\mathbb P_{3}\Theta^*_p
\end{split}
\end{equation}
are projectors on $L^p(G,\R^2)$ and $D(B_p)$, respectively.
We set 
\[
	\mL^p:=\widetilde\mQ_p\left(L^p(G,\R^2)\right)
	=\widetilde\Theta_*^pL^p(\R,E^p_3)
\]
and define the restricted operator
\[
	\mB_p:=B_p|_{D(\mB_p)}
	\quad \text{with}\quad
	D(\mB_p):=\mQ_p\left(D(B_p)\right)
	=\Theta_*^pD_3(T_p).
\]

Notice that, unless its meaning is given otherwise, 
in what follows we understand the 
multiplication operator $M_\alpha v:=e^{\alpha x} v$ for $\alpha\in\R$
as an operator $M_\alpha: F\to M_\alpha(F)$ for a function space $F$.
It is clear that $M_\alpha$ is injective for all 
appearing function spaces $F$. 
Equipping $M_\alpha(F)$ with its canonical norm, we even have
$M_\alpha\in \sLis\left(F,\,M_\alpha(F)\right)$ and
$M_\alpha^{-1}=M_{-\alpha}$. Furthermore, if $T\in \sL(F)$ commutes with 
$M_\alpha$, then we also have $T\in \sL(M_\alpha(F))$.

By construction it follows
\begin{proposition}\label{proppropq}
Let $1<p<\infty$. Then we have
\begin{enumerate}
\item The scale $(\widetilde\mQ_p)_{1<p<\infty}$ is consistent
on $(L^p(G,\R^2))_{1<p<\infty}$ and the scale 
$(\mQ_p)_{1<p<\infty}$ on $(D(B_p))_{1<p<\infty}$;
\item $\widetilde\mQ_p v=\mQ_p v$ for $v\in D(B_p)\cap L^p(G,\R^2)$;
\item $B_p\mQ_p=\widetilde\mQ_pB_p$;
\item $\mB_p\in \sLis\left(D(\mB_p),\,\mL^p\right)$.
\end{enumerate}
In particular, for every $f\in\mL^p$ there is a unique solution 
$u \in D(\mB_p)$ of \eqref{lapequ}.
\end{proposition}
\begin{proof}
(1)\ Obviously we have
\begin{equation}\label{commmap}
	M_\alpha\mP_3v=
	\mP_3M_\alpha v\quad \left(v\in C^\infty_c(\R,D(\cT_{p,\theta})),
	\ \alpha\in\R\right)
\end{equation}
with $\cT_{p,\theta}$ as defined in the beginning of 
Subsection~\ref{sec_prob_transformiert}.
From Lemma~\ref{fullellreg} and Lemma~\ref{denselem} 
we infer that
$C^\infty_c(\R,D(\cT_{p,\theta}))$ is dense in $D(T_p)$.
Thus equality (\ref{commmap}) extends to $v\in D(T_p)$.
By the definition of $\Theta^p_*$ and $\Theta^*_p$ 
(see (\ref{pullback})) this implies
\begin{equation}\label{mqconsequ}
	\mQ_p u
	= \Psi^{-1}\mathcal O M_{\beta_p}\mP_3 M_{-\beta_p}
	\mathcal{O}^{-1}\Psi u
	= \Psi^{-1}\mathcal O \mP_3 
	\mathcal{O}^{-1}\Psi u
	\quad (u\in D(B_p)).
\end{equation}
By the fact that all operators on the right-hand 
side do not depend on $p$ we obtain consistency of
$(\mQ_p)_{1<p<\infty}$. The consistency of 
$(\widetilde\mQ_p)_{1<p<\infty}$ is completely analogous.
\\[2mm]
(2)\ For $u\in D(B_p)\cap L^p(\Omega,\R^2)$ we deduce similarly 
as in (\ref{mqconsequ}) that
\begin{align*}
	\mQ_p u
	&= \Psi^{-1}\mathcal O M_{\beta_p}\mP_3 M_{-\beta_p}
	\mathcal{O}^{-1}\Psi u
	= \Psi^{-1}\mathcal O \mP_3 
	\mathcal{O}^{-1}\Psi u\\
	&= \Psi^{-1}\mathcal O M_{\beta_p+2}\mP_3 M_{-\beta_p-2}
	\mathcal{O}^{-1}\Psi u
	=\widetilde\mQ_p u.
\end{align*}
(3)\ Thanks to Lemma~\ref{lemm_kommutativ} we have
\[
	B_p\mQ_p
	=\widetilde\Theta_*^p\, T_{p}\,\Theta^*_p\,\Theta_*^p\,
	\mathbb P_{3}\,\Theta^*_p
	=\widetilde\Theta_*^p\, \mathbb P_{3}\, T_p\,\Theta^*_p
	=\widetilde\mQ_p B_p.
\]
(4)\ This is a consequence of representation
\[
	\mB_p=\widetilde\Theta_*^p\, T_{p}\,\Theta^*_p
	\quad\text{on}\quad D(\mB_p),
\]
Lemma~\ref{lemma_result}, Corollary~\ref{tpinvss}, 
and the definition of $\mL^p$, $D(\mB_p)$. 
\end{proof}
As for the projector $\mP_3$ before, due to the consistency
we write from now on $\mQ$ and $\widetilde\mQ$, i.e., we omit the 
subscript $p$.

Next, we show consistency of the family $(\mB_p^{-1})_{1<p<\infty}$
on the subscale $(\mL^p)_{1<p<\infty}$. 
Observe that the operator $\mB_p^{-1}$ is represented as 
\begin{equation}\label{repbm1}
	\mB_p^{-1}=\bigl.\Theta_*^pT_p^{-1}\widetilde{\Theta}^*_p\,\bigr|_{\mL^p}.
\end{equation}
So, for consistency we need to prove that the right-hand side
above does not depend on $p$. Note, however, that the single
components $\Theta_*^p$, $T_p^{-1}$, $\widetilde{\Theta}^*_p$ do depend
on $p$. Merely their combination can be consistent. 
For this purpose we first show
\begin{lemma}\label{commuteexp}
Let $1<p\le q<\infty$ and $\beta_p=2-2/p$. 
For $f\in C^\infty_c(\R,E_3^q)$ we have
\[
	T_p^{-1}e^{(\beta_q-\beta_p)x}f=
	e^{(\beta_q-\beta_p)x}T_q^{-1}f.
\]
\end{lemma}
\begin{proof}
First note that $f\in C^\infty_c(\R,E_3^q)$ and $p\le q$ yield
\begin{equation}\label{firstemb}
	e^{(\beta_q-\beta_p)x}f\in C^\infty_c(\R,E_3^q)
	\subset L^p(\R,E_3^p).
\end{equation}
Hence the application of $T_p^{-1}$ to this quantity is defined.
Also recall that 
\[
	T_pv=T_{p,\theta}v+T_{p,x}v
	=\mathcal T_{p,\theta}v+\mathcal T_{p,x}v
	= \mathcal T_{p,\theta}v+(\partial_x+\beta_p)^2v.
\]
We observe that 
\[
	(\partial_x+\beta_q)^2e^{-(\beta_q-\beta_p)x}
	=e^{-(\beta_q-\beta_p)x}(\partial_x+\beta_p)^2
\]
implies that 
\begin{equation}\label{preext}
	e^{(\beta_q-\beta_p)x}T_qe^{-(\beta_q-\beta_p)x}v=T_pv
	\quad \left(v\in C^\infty_c(\R,D_3(\mathcal T_{\theta,q}))\right),
\end{equation}
as an equality in $C^\infty_c(\R,E_3^p)$. 
Here we set $D_3(\mathcal T_{\theta,q})=D(\mathcal T_{\theta,q})\cap E_3^q$
and notice that the assertions of Corollary~\ref{lemprojdoms}
also hold for $\mathcal T_{\theta,q}$.

For $v\in C^\infty_c(\R,D_3(\mathcal T_{\theta,p}))
\hook C^\infty_c(\R,E^q_3)$ (Sobolev embedding) we set 
\begin{equation}\label{approxvci}
	v_k:=k(k+T_{q,\theta})^{-1}v 
	\in C^\infty_c(\R,D_3(\mathcal T_{\theta,q})),\quad k\in\N.
\end{equation}
By the sectoriality of $T_{q,\theta}$ we obtain 
$v_k\to v$ in $D_3(T_p)$.
Hence equality (\ref{preext}) extends to 
$v\in C^\infty_c(\R,D_3(\mathcal T_{\theta,p}))$.
Setting $X=D_3(\mathcal T_{\theta,p})$, $Y=E_3^p$, $k=0$,
and $\ell=2$ in Lemma~\ref{denselem}, we see that (\ref{preext})
extends to all $v\in D_3(T_p)$.

As before, for $\alpha\in\R$ we set
$M_\alpha v=e^{\alpha x} v$. 
For $\alpha=\beta_q-\beta_p$ 
relation (\ref{preext}) then yields
\[
	T_q=M_{-\alpha}T_pM_\alpha
	\in\sLis\bigl(M_{-\alpha}(D_3(T_p)),
	\,M_{-\alpha}(L^p(\R,E_3^p))\bigr)
\]
with inverse
\[
	\widetilde T_q^{-1}=M_{-\alpha}T_p^{-1}M_\alpha.
\]
Thanks to (\ref{firstemb}) we see that 
\[
	f=M_{-\alpha}\left(M_\alpha f\right)
	\in M_{-\alpha}\left(L^p(\R,E_3^p)\right)
\]
for $f\in C^\infty_c(\R,E_3^q)$. Due to this fact
it remains to show that $\widetilde T_q^{-1}$ is consistent with
$T_q^{-1}$ on $C^\infty_c(\R,E_3^q)$.

For $f\in C^\infty_c(\R,D_3(\mathcal T_{\theta,q}))$ we have 
$M_\alpha f\in C^\infty_c(\R,E_3^p)$ and hence 
$T_p^{-1}M_\alpha f\in D_3(T_p)$. Since (\ref{preext}) 
holds for all $v\in D_3(T_p)$ this yields
\[
	T_q\widetilde T_q^{-1}f
	=T_qM_{-\alpha}T_p^{-1}M_\alpha f
	=M_{-\alpha}\underbrace{M_\alpha T_q M_{-\alpha}}_{=T_p}
	 T_p^{-1}M_\alpha f=f.
\]
Completely analogous we deduce $\widetilde T_q^{-1}T_qf=f$
for $f\in C^\infty_c(\R,D_3(\mathcal T_{\theta,q}))$. 
Hence $\widetilde T_q^{-1}=T_q^{-1}$ on the set 
$C^\infty_c(\R,D_3(\mathcal T_{\theta,q}))$.
By a similar approximation argument as in (\ref{approxvci})
we see that this consistency extends to
$C^\infty_c(\R,E^q_3)$.
This finally yields the assertion.
\end{proof}
In the proof of consistency we also employ the following
density pro\-perty.
\begin{lemma}\label{denselem2}
Let $1<p\le q<\infty$. Then we have
\[
	\widetilde{\Theta}_*^q\bigl(C^\infty_c(\R,E^q_3)\bigr)
	\hookd \mL^q\cap\mL^p.
\]
\end{lemma}
\begin{proof}
Note that
\[
	\mL^p=\widetilde{\Theta}_*^p\bigl(L^p(\R,E^p_3)\bigr)
	=\widetilde{\Theta}_*^qM_{-\alpha}\bigl(L^p(\R,E^p_3)\bigr)
\]
with $M_{-\alpha}$ as defined in the proof of 
Lemma~\ref{commuteexp} and where $M_{-\alpha}\bigl(L^p(\R,E^p_3)\bigr)$ 
is again equipped with its canonical norm. This shows that 
$\widetilde{\Theta}_*^q\in
\sLis\left(M_{-\alpha}\bigl(L^p(\R,E^p_3)\bigr),\, \mL^p\right)$ with
inverse $\widetilde{\Theta}^*_q$. Since $\widetilde{\Theta}_*^q\in
\sLis\left(L^q(\R,E^q_3),\, \mL^q\right)$ has the same inverse
we conclude that
\[
	\widetilde{\Theta}_*^q\in
	\sLis\biggl(L^q(\R,E^q_3)\cap M_{-\alpha}\bigl(L^p(\R,E^p_3)\bigr),\, 
	\mL^q\cap \mL^p\biggr).
\]
Thus, it suffices to show that 
\[
	C^\infty_c(\R,E^q_3)
	\ \hookd \ L^q(\R,E^q_3)\cap
	M_{-\alpha}\bigl(L^p(\R,E^p_3)\bigr)=:Y.
\]

To this end, pick $v\in Y$ and choose a bounded interval $J\subset\R$ such that
\[
	\|v-\chi_{J}v\|_Y
	=\|v-\chi_{J}v\|_{L^q(\R,E^q_3)}
	+\|M_{\alpha}(v-\chi_{J} v)\|_{L^p(\R,E^p_3)}<\eps/2,
\]
where $\chi_{J}$ denotes the characteristic function to $J$.
By the fact that $\chi_{J}v\in L^q(J,E^q_3)$ we find 
$(v_k)\subset C^\infty_c(J,E^q_3)$ such that $v_k\to \chi_{J}v$
in $L^q(\R,E^q_3)$. Note that, thanks to $E^q_3\hook E^p_3$, we also have
\[
	\|M_{\alpha}(\chi_Jv-v_k)\|_{L^p(\R,E^p_3)}
	\le C(J,\alpha)\|\chi_Jv-v_k\|_{L^p(J,E^q_3)}\to 0
	\quad (k\to\infty).
\]
Consequently, choosing $k$ large enough we can achieve
\[
	\|v-v_k\|_Y\le \|v-\chi_Jv\|_Y+\|\chi_Jv-v_k\|_Y<\eps
\]
and the assertion is proved.
\end{proof}
Now we are in position to prove the claimed consistency.
\begin{proposition}\label{konsitentbp}
The family $(\mB_p^{-1})_{1<p<\infty}$ is consistent
on the subscale $(\mL^p)_{1<p<\infty}$.
\end{proposition}
\begin{proof}
Let $p,q\in(1,\infty)$ and without loss of generality $p\le q$.
By the definition of $\Theta_*^p, \ \widetilde{\Theta}^*_p  $ we have
\begin{equation*}
\Theta_*^p = \Theta_*^{q}e^{-(\beta_q - \beta_{p})x} \quad \text{and}
\quad \widetilde{\Theta}_p^* = e^{(\beta_q - \beta_{p})x}
\widetilde{\Theta}^*_{q}.
\end{equation*}
Now, pick 
\[
	f\in \widetilde{\Theta}_*^q\left(C^\infty_c(\R,E^q_3)\right)
	\subset \mL^p\cap\mL^q.
\]
From (\ref{repbm1}) and Lemma~\ref{commuteexp} we infer
\begin{align*}
	\mB_p^{-1} f 
	&=\Theta_*^{p}\,T_p^{-1}\,\widetilde{\Theta}^*_{p}f\\
	&= \Theta_*^{q}\, e^{-(\beta_q - \beta_{p})x} \, 
	T_p^{-1} \, e^{(\beta_q - \beta_{p})x} \, 
	\widetilde{\Theta}^*_{q}f\\
	&=\Theta_*^{q}\, T_q^{-1}\, 
	\widetilde{\Theta}^*_{q}f
	=\mB_q^{-1}f. 
\end{align*}
Proposition~\ref{proppropq}(4) 
and Lemma~\ref{denselem2} then yield the assertion.
\end{proof}

\subsection{The diffusion equation} 
\label{secreglap}

As before let $\theta_0 \in (0, \pi)$ be the opening angle of the wedge
$G$. For $1<p<\infty$ we define the Laplacian $A_p$ subject to 
perfect slip boundary conditions in $L^p(G,\R^2)$ by 
\begin{equation}
	\begin{aligned}	
		 A_p u&:= - \Delta u,\\
		 u\in D(A_p) 
		 &:=   \left\{ u \in W^{2,p}(G, \mathbb{R}^2): 
		 \ \textrm{curl} \ u=0, \ \nu \cdot u=0 \ \text{on} \ 
		 \partial G \right\}
		 \cap K^2_p(G,\R^2).
	\end{aligned}
	\label{operator_ap}
\end{equation} 
Now \cite[Theorem 1.1 and Corollary 3.15]{Maier} gives the following
result.
\begin{theorem}\label{thmmasa}
There is a $\delta=\delta(\theta_0)$ such that for 
$1<p<1+\delta$ the operator $A_p$ as defined in (\ref{operator_ap})
has maximal regularity on 
$L^p(G,\R^2)$. 
\end{theorem}
\begin{remark}\label{remthmmasa}
(a)\ Note that in \cite{Maier} the case of a three-dimensional wedge
is considered. However, by an inspection of the single steps
in the proof it is clear that the case of
a two-dimensional wedge is completely analogous.\\[1mm]
(b)\ Also observe that $\delta>0$ can be very small. In fact,
the methods in \cite{Maier} yield the constraint
$2-2/p<\min\{1,(\pi/\theta_0-1)\}$. 
Hence we have $\delta(\theta_0)\to 0$ for $\theta_0\to\pi$.\\[1mm]
(c)\ From the proof of \cite[Theorem 1.1 and Corollary 3.15]{Maier}
it also follows that for each $\lambda\in\rho(A_p)$
the family $\left((\lambda-A_p)^{-1}\right)_{1<p<1+\delta}$ is
consistent on $\left(L^p(G,\R^2)\right)_{1<p<1+\delta}$.
\end{remark}
By a scaling argument we obtain the following estimate in
the homogeneous norm.
\begin{lemma}\label{scallem}
Let $1<p<\infty$ and $\rho(A_p)\neq\emptyset$. Then we have 
\[
	\|u\|_{K_p^2(G,\R^2)}
	\le C\|A_p u\|_{L^p(G,\R^2)}
	\quad (u\in D(A_p)).
\]
\end{lemma}
\begin{proof}
We have $\mu-A_p\in\sLis\left(D(A_p),\,L^p(G,\R^2)\right)$
for a $\mu\in\C$.
We introduce the rescaled function $J_\lambda u(x):=\lambda^{-2}u(\lambda x)$,
$\lambda>0$, and note that the wedge $G$ is invariant under this 
scaling. This yields
\begin{align*}
	\|u\|_{K_p^2(G,\R^2)}
	&= \lambda^{2/p}\|J_\lambda u\|_{K_p^2(G,\R^2)}
	\le C\lambda^{2/p}\|(A_p-\mu)J_\lambda u\|_{L^p(G,\R^2)}\\
	&\le
	C\lambda^{2+2/p}\|J_\lambda(A_p-\mu\lambda^{-2})u\|_{L^p(G,\R^2)}\\
	&= C\|(A_p-\mu\lambda^{-2})u\|_{L^p(G,\R^2)}
	\quad (\lambda>0,\ u\in D(A_p)).
\end{align*}
Letting $\lambda\to \infty$ yields the assertion.
\end{proof}
\begin{remark}\label{scalrem}
The estimate in Lemma~\ref{scallem} implies that $A_p$ is injective
provided that $\rho(A_p)\neq\emptyset$. This implies that
$A_p$ is sectorial or $\cR$-sectorial, whenever 
$(\lambda(\lambda+A_p)^{-1})_{\lambda>0}$ is uniformly bounded or 
$\cR$-bounded, respectively, see \cite{Haase}.
\end{remark}

Next, we show that Theorem~\ref{thmmasa} is still valid on 
$\mL^p$.
To this end, for $1<p<\infty$ we define $\mA_p$ as 
the part of $A_p$ in $\mL^p$, that is
\begin{align*}
	\mA_pu := A_p|_{\mL^p}u,&\quad
	u\in D(\mA_p):= \left\{v\in D(A_p)\cap \mL^p:
	\ A_p u\in \mL^p\right\}.
\end{align*}
With the projectors $\mQ$ and $\widetilde\mQ$ as defined in
(\ref{defprojq}) we obtain
\begin{lemma}\label{lempartap}
Let $1<p<\infty$. We have
\begin{enumerate}
\item $D(A_p)= D(B_p)\cap L^p(G,\R^2)$ with equivalent norms 
as well as $\mQ=\widetilde\mQ$ 
and $A_p=B_p$ on $D(A_p)$. In particular, $\widetilde\mQ$
on $L^p(G,\R^2)$ is the continuous extension of
$\mQ$ regarded as a projector on $D(A_p)$.
\item $\widetilde \mQ A_p u=A_p \mQ u$ for $u\in D(A_p)$.
\item $\mQ(\lambda-A_p)^{-1}f=(\lambda-A_p)^{-1}\widetilde\mQ f$
for $f\in L^p(G,\R^2)$ and $\lambda\in\rho(A_p)$.
\item $D(\mA_p)=D(A_p)\cap \mL^p=\mQ D(A_p)$.
\item $(\lambda-\mA_p)^{-1}
	=(\lambda-A_p)^{-1}|_{\mL^p}$
	for $\lambda\in\rho(A_p)$.
\item $\left((\lambda-\mA_p)^{-1}\right)_{1<p<1+\delta}$ is
consistent on $\left(\mL^p\right)_{1<p<1+\delta}$
for $\lambda\in\rho(\mA_p)$.
\end{enumerate}
\end{lemma}
\begin{proof}
(1) Note that $D(A_p)\hook D(B_p)$ is an immediate consequence 
of the definition of $D(A_p)$. This gives $B_p=A_p$
and, by virtue of Proposition~\ref{proppropq}(2), also 
$\mQ=\widetilde\mQ$ on $D(A_p)$.
Furthermore, the Gagliardo-Nirenberg inequality and Young's inequality 
yield
\[
	\|\nabla u\|_p\le C\left(\|\nabla^2 u\|_p+\|u\|_p\right)
	\quad (u\in L^p(G,\R^2)\cap K^2_p(G,\R^2)).
\]
Note that the wedge $G$ is an $(\eps,\infty)$ domain and on domains
of this type the Gagliardo-Nirenberg inequality holds true 
\cite[Section~5]{Maier} thanks to the extension operator 
for homogeneous Sobolev spaces constructed
in \cite{jones81,chua92}. This implies
\[
	\|u\|_{W^{2,p}}
	\le C\left(\|u\|_p+\|\nabla^2 u\|_p\right)
	\le C\left(\|u\|_p+\| u\|_{K^2_p}\right).
\]
Thus $D(A_p)= D(B_p)\cap L^p(G,\R^2)$ with equivalent norms.
From this we easily obtain that $\mQ$ is also a projector on $D(A_p)$.
Since $D(A_p)$ is dense in $L^p(G,\R^2)$, $\widetilde\mQ$ extends
$\mQ$ continuously on $L^p(G,\R^2)$.
\\[1mm] 
(2) follows directly from (1) and Proposition~\ref{proppropq}(3).\\[1mm] 
(3)\ Let $\lambda\in\rho(A_p)$. From (1) and (2) we obtain
\[
	(\lambda-A_p)\mQ(\lambda-A_p)^{-1}f
	=\widetilde\mQ f \quad (f\in L^p(G,\R^2)).
	\]
	Applying $(\lambda-A_p)^{-1}$ on both sides yields (3).\\[1mm]
	(4)\ Let $u\in D(A_p)\cap\mL^p$. By (1) we obtain $u=\widetilde\mQ u
	=\mQ u$, hence $u\in\mQ D(A_p)$. Conversely, 
	(1) also yields
$\mQ D(A_p)\subset D(A_p)\cap \mL^p$. In view of (2) 
	we next conclude
	\[
		A_p u=A_p\mQ u=\widetilde \mQ A_pu\in \mL^p,
		\]
		hence $u\in D(\mA_p)$. Since the inclusion 
		$D(\mA_p)\subset D(A_p)\cap \mL^p$ is trivial, the assertion
	is proved.\\[1mm]
(5)\ Let $\lambda\in\rho(A_p)$. For $f\in\mL^p$ relations (3) 
	and (4) yield
	\[
		(\lambda -A_p)^{-1}f
		=\mQ (\lambda -A_p)^{-1}f\in D(\mA_p).
		\]
		Thus,
		\[
			(\lambda-\mA_p)(\lambda-A_p)^{-1}f=f
			\]
			which proves (5).\\[1mm]
			(6)\ follows from (5) and Remark~\ref{remthmmasa}(c).
			\end{proof}

By combining the well-known equivalence of maximal regularity and
$\cR$-sectoriality \cite[Theorem 4.2]{Weis} with Theorem~\ref{thmmasa},
Remark~\ref{scalrem}, and Lemma~\ref{lempartap} (especially assertion
(5)) we obtain
\begin{theorem}\label{thmmasapart}
Let $1<p<1+\delta$ with $\delta>0$ as in Theorem~\ref{thmmasa}. 
Then $\mA_p:D(\mA_p)\to\mL^p$ with domain
\[
	D(\mA_p) =   \left\{ u \in W^{2,p}(G, \mathbb{R}^2):\,
	\mathrm{curl}\, u=0,\, \nu \cdot u=0\,
	\text{on}\, \partial G \right\}\cap K^2_p(G,\R^2)\cap \mL^p
\]
is $\cR$-sectorial with $\phi^\cR_{\mA_p}<\pi/2$. Thus,
$\mA_p$ has maximal regularity on 
$\mL^p$. 
\end{theorem}

Our ultimate aim in this subsection is to show that 
Theorem~\ref{thmmasapart}, in particular the optimal Sobolev 
regularity, is available on the full range $1<p<\infty$.
Note that this is not true for $A_p:D(A_p)\subset L^p(G,\R^2)\to 
L^p(G,\R^2)$ with $D(A_p)$ given in (\ref{operator_ap}) 
as the next result shows.
\begin{theorem}\label{thmmasneg}
Let $1<p<\infty$ and $\theta_0\in(0,\pi)$ such that
condition (\ref{spec_cond}) (with $\gamma=0$) is not satisfied.
Then $\rho(A_p)=\emptyset$. In other words, in this situation
for every $\lambda\in\C$ there is an $f\in L^p(G,\R^2)$ 
such that there is no solution $u$ of
\begin{equation}
	 	\left.
	 	\begin{array}{r@{\ =\ }lll}
	 		\lambda u- \Delta u &  f &
			\text{in} & G, \\
	 		\text{curl}\, u=0, \ u \cdot \nu &0 &
			\text{on} & \partial G 
	 	\end{array}
	 	\right\}
	 \label{lapresequ}
\end{equation}
satisfying $u\in K^2_{p}(G,\R^2)$. More precisely, if $p \neq 2$ then $\partial^\alpha u \notin L^p(G, \bR^2)$ for some $\alpha$ with $|\alpha| = 2$,
while for $p = 2$ we have $\rho^{|\alpha| - 2} \partial^\alpha u \not\in L^2(G, \bR^2)$ for some $\alpha$ with $|\alpha| < 2$.
\end{theorem}
\begin{proof}
Suppose there exists a complex number $\mu\in\rho(A_p)$. 
We can assume $\mu\neq 0$,
since otherwise this would immediately contradict 
Theorem~\ref{main1}.

By the scaling argument used in the proof of
Lemma~\ref{scallem} it easily follows that
$((\lambda-A_p/\mu)^{-1})_{\lambda>0}$ is uniformly bounded.
Thanks to Remark~\ref{scalrem} then $A_p/\mu$ is sectorial,
see \cite{Haase}, in particular it has dense range. For 
$f\in L^p(G,\R^2)$ we hence find $(u_k)\subset D(A_p)$ such that
$A_pu_k\to f$ in $L^p(G,\R^2)$. Due to Lemma~\ref{scallem} $(u_k)$
is a Cauchy sequence in $K_p^2(G,\R^2)$ and its limit $u=\lim u_k$ 
satisfies equation (\ref{lapequ}). The fact that
$u\in K_p^2(G,\R^2)$ then contradicts Theorem~\ref{main1}. 
Thus $\rho(A_p)$ must be empty.
The additional statement follows from Remark~\ref{remconsnotw2}(b).
\end{proof}

Next, we show that the resolvent of $\mA_p$ in $\mL^p$ 
is consistent with its dual resolvent. 
For this purpose we first identify $(\mL^p)'$. This, in turn,
is connected to the identification of $\mP_3'$ and $\mQ'$. By this fact,
just within the following lemma, we write $\mP_{3,p}$ and $\mQ_p$ again.
\begin{lemma}\label{dualqmlp}
Let $1<p<\infty$, $\beta_p=2-2/p$, and $1/p+1/p'=1$. Let  
$\widetilde{\Theta}_*^p:L^p(\Omega,\R^2)\to L^p(G,\R^2)$ be defined as
in Subsection~\ref{first} with inverse $\widetilde{\Theta}^*_p$ 
and the projectors $\mP_{3,p}$ and $\widetilde\mQ_p$
be defined as in (\ref{defpmproj}) (and the subsequent lines) 
and (\ref{defprojq}) respectively. Then we have
\begin{enumerate}
\item $(\widetilde{\Theta}_*^p)'=\widetilde{\Theta}_{p'}^*$ and 
$(\widetilde{\Theta}^*_p)'=\widetilde{\Theta}^{p'}_*$;
in particular $\widetilde{\Theta}_*^p$ is an isometric isomorphism;
\item $(\mP_{3,p})'=\mP_{3,p'}$;
\item $(\widetilde\mQ_{p})'=\widetilde\mQ_{p'}$;
\item $(\mL^p)'=\mL^{p'}$ with respect to $(u,\,v)=\int_Guv dx$
in the sense of a Riesz isomorphism.
\end{enumerate}
\end{lemma}
\begin{proof}
(1)\ Recall that by (\ref{pullback})
and (\ref{substitution}) we have 
$\widetilde{\Theta}_*^p u=\Psi^{-1}\mathcal O M_{\beta_p-2} u$
with $\Psi$, $\mathcal O$, $M_{\beta_p-2}$ 
as defined in Subsection~\ref{first}. Thanks to
\[
	\beta_p=2-\frac{2}{p}=-\beta_{p'}+2
\]
we can calculate 
\begin{align*}
	\left(\widetilde{\Theta}_*^{p}u,\, v\right)_G
 	&= \int_G v(y)(\mathcal O M_{\beta_p-2}u)(\psi^{-1}(y)) dy\\
	&= \int_\Omega v(\psi(x,\theta))
		(\mathcal O M_{\beta_p-2}u)(x,\theta) e^{2x}dxd\theta\\
	&= \int_\Omega (M_{-\beta_{p'}+2}\mathcal O^{-1}\Psi v)(x,\theta)
			   u(x,\theta) dxd\theta\\
	&= \left(u,\, \widetilde{\Theta}^*_{p'}v\right)_\Omega
	\quad \left(u\in L^p(\Omega,\R^2),\ v\in L^{p'}(G,\R^2)\right).
\end{align*}
Relation $(\widetilde{\Theta}^*_p)'=\widetilde{\Theta}^{p'}_*$ then follows
since $\widetilde{\Theta}^*_p=(\widetilde{\Theta}_*^p)^{-1}$.

Relation (2) follows immediately by the definition of $\mP_{3,p}$ and 
(3) is a consequence of (1) and (2).

(4)\ By the fact that $\mL^p=\widetilde\mQ_p L^p(G,\R^2)$ this follows
from the symmetry of $\widetilde\mQ_p$ proved in (3)
and since $(L^p(G,\R^2))'=L^{p'}(G,\R^2)$ with 
respect to $(\cdot,\ \cdot)$.
\end{proof}

Now, let 
\[
	\mA_p':D(\mA_p')\subset\mL^{p'}\to \mL^{p'}
\]
be the Banach space dual operator to $\mA_p$ in $\mL^p$ for $1<p<1+\delta$.
By permanence properties and Theorem~\ref{thmmasapart} 
it follows that also $\mA_p'$ is $\cR$-sectorial with 
$\phi^\cR_{\mA_p'}=\phi^\cR_{\mA_p}<\pi/2$.
At this point, however, we do not
know how $D(\mA_p')$ looks like. On our way to characterize $D(\mA_p')$
we next show consistency of
$(\lambda-\mA_p)^{-1}$ and $(\lambda-\mA_p')^{-1}$ on $\mL^p\cap\mL^{p'}$. 
\begin{proposition}\label{konsistentap}
Let $1 < p < 1+ \delta$ with $\delta>0$ as in Theorem~\ref{thmmasapart}
and $1/p + 1/ p'=1$. Then 
$$ (\lambda -\mA_p)^{-1} f = (\lambda -\mA_p')^{-1} f \quad (f \in 
		\mL^p\cap\mL^{p'},\ \lambda\in\rho(\mA_p)\cap\R). $$
\end{proposition}
\begin{proof}
Let $\lambda\in\rho(\mA_p)\cap\R$. We intent to apply Lemma~\ref{conslem}.
Setting $T=\lambda-\mA_p$, we first have to
verify that there exists an embedding $J:D(\mA_p)\to (\mL^p)'$
with dense range.
Observe that, since $D(\mA_p)\hook W^{2,p}(G,\R^2)$ and $G\subset\R^2$, 
the Sobolev embedding yields 
\[
	D(\mA_p)\hookd L^{p'}(G,\R^2)\cap \mL^p=\mL^{p'}.
\]
Thus $J$ can be chosen essentially as the Riesz isomorphism given
in Lemma~\ref{dualqmlp}(4). However, since we identify $(\mL^p)'$
with $\mL^{p'}$ anyway and $T^\sharp$ with $(\lambda-\mA_p)^\sharp$ 
on $\mL^{p'}$, that is, with its dual induced by 
the Riesz isomorphism,
we omit $J$ (and hence also $\widetilde J$) in what follows.

By virtue of Lemma~\ref{conslem} and (\ref{dualcons})
it then remains to prove that
\[
	\lambda-\mA_p\subset (\lambda-\mA_{p})^\sharp,
\]
where $(\lambda-\mA_{p})^\sharp:\mL^{p'}\to D(\mA_{p})'$ denotes the dual
operator of $\lambda-\mA_{p}$ regarded as a bounded operator from
$D(\mA_{p})$ to $\mL^{p}$, see Appendix~\ref{appa}.
To this end, pick $u,v\in D(\mA_p)$. Observe that by the fact that 
$D(\mA_p)\hook \mL^p\cap\mL^{p'}$ all duality pairings appearing below 
are well-defined. Also note that 
$$ \Delta u = \nabla \text{div} \ u - \text{curl}' \, \text{curl} \ u, $$
where $\text{curl}'\varphi = \left( \partial_{x_2}, - \partial_{x_1}
\right)^T \varphi$ for a scalar function $\varphi$. 
Employing the Gau{\ss} theorem and the boundary conditions for $u$ and
$v$ we calculate
\begin{equation*}
\left( \nabla \, \text{div} \, u, v \right)
	= \int_{\partial G} \nu \cdot v\ \text{div}\, u\,d \sigma 
- \left(\text{div}\,u,\, \text{div}\,v \right) 
= \left(u,\, \nabla \, \text{div} \, v\right) 
\end{equation*} 
as well as
\begin{equation*}
\begin{aligned}
\left(\text{curl}' \, \text{curl} \, u, v \right)
&= -\int_{\partial G}\left( \left(v^2, -v^1  \right)^T \cdot \nu			\right) \text{curl}\,u\, d \sigma + \left(\text{curl}\,u,\,
				\text{curl}\,v\right) \\
	&= \left(u,\, \text{curl}' \, \text{curl} \, v \right).
\end{aligned}
\end{equation*}
This yields
\begin{equation*}
\begin{aligned}
\langle T^\sharp u,\, v \rangle_{D(\mA_p)',D(\mA_p)} 
&= \left( u,\, (\lambda+ \Delta) v \right)
	= \left((\lambda+ \Delta)  u,\,v\right)\\ 
&= (Tu,\,v)
=\langle T u,\, v \rangle_{D(\mA_p)',D(\mA_p)}
\end{aligned}
\end{equation*} 
which proves the claim. 
\end{proof}

Now we can characterize $D(\mA_p')$.
\begin{theorem}\label{korollar_ergebnis_paper}
Let $1<p< \infty$ and $1/p+1/p'=1$. Then we have $\mA_p'=\mA_{p'}$,
i.e., in particular $D(\mA_p')=D(\mA_{p'})$ with $D(\mA_{p'})$ as
characterized by (\ref{operator_ap}) and 
Lemma~\ref{lempartap}(4). Furthermore, for $\lambda\in \rho(\mA_p)$
the family $\left((\lambda-\mA_p)^{-1}\right)_{1<p<\infty}$
is consistent on $\left(\mL^p\right)_{1<p<\infty}$.
\end{theorem}
\begin{proof}
By definition it is obvious that $\mA_{p'}\subset \mA_p'$. It is clear
that the converse inclusion, particularly the assertion on $D(\mA_p')$,  
is proved, if we can show that
\begin{equation}\label{resmaiso}
	(1+\mA_p)^{-1}\in\sLis(\mL^{p},\,D(\mA_{p}))
\end{equation}
for every $p\in(1,\infty)$.
By Theorem~\ref{thmmasapart} relation (\ref{resmaiso}) holds
for every $1<p<1+\delta$. We take $p$ out of that interval 
and consider (\ref{resmaiso}) for its H\"older conjugated exponent 
$p'$.  

Let $f\in \mL^p\cap\mL^{p'}$. Then there is a $u\in D(\mA_p')$ such that
\[
		(1+\mA_p')u=f.
\]
By the consistency of the resolvents of $\mA_p$ and $\mA_p'$ 
proved in Proposition~\ref{konsistentap} we see that $u\in D(\mA_p)$
and that 
\[
	(1+\mA_p)u=f \quad\Leftrightarrow\quad 
	\mA_pu=f-u=:g\in \mL^p\cap\mL^{p'}.
\]
On the other hand,
Proposition~\ref{proppropq}(4) and the consistency of
$(\mB_p^{-1})_{1<p<\infty}$ established in Proposition~\ref{konsitentbp}
imply that there is an $v\in D(\mB_p)\cap D(\mB_{p'})$ such that
\[
	\mB_p v=g.
\]
The fact that $D(\mA_p)\subset D(\mB_p)$ and $\mA_p=\mB_p$ on
$D(\mA_p)$ (Lemma~\ref{lempartap}(1),(4)) then gives $u=v$.
From this and Lemma~\ref{lempartap}(1) we obtain
\begin{align*}
	\|(1+\mA_p')^{-1}f\|_{D(\mA_{p'})}
	&=\|u\|_{D(\mA_{p'})}
	\le C\left(\|u\|_{p'}+\|v\|_{K^2_{p'}}\right)\\
	&\le C\|f\|_{p'}
	\quad (f\in \mL^p\cap\mL^{p'}).
\end{align*}
Since $\mL^p\cap\mL^{p'}$ lies dense in $\mL^{p'}$, relation
(\ref{resmaiso}) follows for $p'$. 

According to what we just have proved, Lemma~\ref{lempartap}(6), 
and Proposition~\ref{konsistentap}
the family
$\left((1+\mA_{p})^{-1}\right)_{p\in I}$ is consistent on
$(\mL^p)_{p\in I}$ for 
\begin{equation}\label{intdualp}
	I=(1,\infty)\setminus[1+\delta,(1+\delta)'].
\end{equation}
For the remaining $p$ we interpolate. In fact, since
$\mL^p=\widetilde\mQ L^p(G,\R^2)$ complex
interpolation and \cite[Theorem~1.17.1.1]{triebel} yield
\[
	\bigl[\mL^p,\mL^{p'}\bigr]_s=\mL^q,
	\quad \frac1q=s\frac1{p'}+(1-s)\frac1{p}.
\]
Furthermore, by \cite{triebel} we also have
\begin{align*}
	W^{2,q}(G,\R^2)&=\bigl[W^{2,p}(G,\R^2),W^{2,p'}(G,\R^2)\bigr]_s,\\ 
	K^2_{q}(G,\R^2)&=\bigl[K^2_p(G,\R^2),K^2_{p'}(G,\R^2)\bigr]_s.
\end{align*}
(Note that the second identity above follows, e.g., from 
\[
	W^{2,q}(\Omega,\R^2)
	=\bigl[W^{2,p}(\Omega,\R^2),W^{2,p'}(\Omega,\R^2)\bigr]_s,\\ 
\]
and an application of Stein's interpolation theorem \cite{voigt1992},
since the dependence of $\Theta^q_*,\Theta^*_q$ on
$z=1/q$ is analytic on a suitable strip in the complex plane.)
This shows that
\[
	(1+\mA_{p})^{-1}\in\sL\left(\mL^p,\,W^{2,q}(G,\R^2)
	\cap K^2_{q}(G,\R^2)\cap\mL^p\right)
\]
for every $p\in(1,\infty)$. For $f\in\mL^p\cap \mL^q$ with 
$q\in I$, we also see that $(1+\mA_{p})^{-1}f$ satisfies the 
boundary conditions included in $D(\mA_p)$. By a density argument 
and boundedness of the corresponding trace operators 
relation (\ref{resmaiso}) follows to be valid for all $p\in(1,\infty)$.
This completes the proof.
\end{proof}
Thanks to Theorem~\ref{korollar_ergebnis_paper} we can generalize
Theorem~\ref{thmmasapart} to all $p\in(1,\infty)$.
\begin{theorem}\label{mainlap}
Let $1<p<\infty$. 
Then $\mA_p$ with domain
\[
	D(\mA_p) =   \left\{ u \in W^{2,p}(G, \mathbb{R}^2):\,
	\mathrm{curl}\, u=0,\, \nu \cdot u=0\,
	\text{on}\, \partial G \right\}\cap K^2_p(G,\R^2)\cap \mL^p
\]
is $\cR$-sectorial on $\mL^p$ with $\phi^\cR_{\mA_p}<\pi/2$, 
and hence has maximal regularity on $\mL^p$. 
\end{theorem}
\begin{proof}
Due to $\mA_p'=\mA_{p'}$ and Theorem~\ref{thmmasapart}, 
the operator $\mA_p$ with $D(\mA_p)$ as stated is $\cR$-sectorial with
$\phi^\cR_{\mA_p}<\pi/2$
for $p\in I$ with $I$ given in (\ref{intdualp}).
Note that injectivity, hence also $\overline{R(\mA_p)}=\mL^p$,
follows from Remark~\ref{scalrem}.
Since the property of $\mathcal R$-sectoriality is invariant under 
interpolation \cite[Theorem 3.23]{Kaip}, the result follows by interpolation
and the equivalence of maximal regularity and $\mathcal R$-sectoriality
\cite[Theorem 4.2]{Weis}.
\end{proof}

In this subsection we have shown by consistency 
arguments that regularity for the elliptic operator $\mB_p$ 
transfers to the parabolic operator $\partial_t+\mA_p$.  
The next result, which in principle shows that the converse is true 
as well, we state also for later purposes.
\begin{proposition}\label{mainlapconv}
Let $1<p<\infty$. If $(\lambda_k)\subset \rho(\mA_p)$ such that
$\lim_{k\to\infty}\lambda_k=0$, then
\[
	\lim_{k\to\infty}(\lambda_k-\mA_p)^{-1}=\mB_p^{-1}
	\quad\text{in}\quad \sL\left(\mL^p,\, K_p^2(G,\R^2)\right). 
\]
In particular, $D(\mA_p)$ is dense in $D(\mB_p)$. 
\end{proposition}
\begin{proof}
Pick $f\in \mL^p$.
For $\ell\in\N$ by the resolvent identity, Lemma~\ref{scallem}, and
since $\mA_p$ is sectorial we obtain
\begin{align*}
	&\|(\lambda_{k+\ell}-\mA_p)^{-1}f-(\lambda_k-\mA_p)^{-1}f\|_{K_p^2}\\
	&\le C\|(\lambda_{k+\ell}-\lambda_k)(\lambda_k-\mA_p)^{-1}
	\mA_p(\lambda_{k+\ell}-\mA_p)^{-1}f\|_{p}\\
	&\le C\|(\lambda_{k+\ell}/\lambda_k-1)f\|_{p}\to 0
	\quad (k\to\infty). 
\end{align*}
Thus $(\lambda_k-\mA_p)^{-1}f\to v$ in $D(\mB_p)$. The fact that
$\mB_p\in\sL(D(\mB_p),\,\mL^p)$, Lemma~\ref{lempartap}(1),
and again sectoriality of $\mA_p$ yield
\[
	\mB_p v=\lim_{k\to\infty}\mA_p(\lambda_k-\mA_p)^{-1}f=f,
\]
hence $v=\mB_p^{-1}f$.
\end{proof}

\section{The Stokes equations}\label{secstokes}

In this section, we consider the Stokes problem
\eqref{stokes_gleichung}. We introduce the space of solenoidal vector
fields. For $1<p< \infty$ and $1/p + 1/ p'=1$ we set
\begin{equation}\label{defsolss}
	L^p_\sigma(G):= \left\{ u \in L^p(G, \mathbb{R}^2): \ \int_G u
	\cdot \nabla \varphi\, d(x_1,x_2) =0 \ \ ( \varphi \in
	\widehat{W}^{1,p'}(G)
	) \ \right\},  
\end{equation}
where
\begin{equation}\label{defsolss2}
\widehat{W}^{1,p'}(G):= \left\{ \varphi \in  L^1_{loc}(G): \ \nabla \varphi \in L^{p'}(G, \mathbb{R}^2) \right\}.
\end{equation}
Since $C_c^\infty (G) \subset \widehat{W}^{1,p}(G)$, 
it is evident that $u \in L_\sigma^p(G)$ satifies the condition
$\text{div}\,u =0$ in the sence of distributions. Moreover $ \nu \cdot
u$ is well-defined in the trace space (Slobodeckii space)
$W^{-1/p}_p(\cO)$ for all bounded domains $\cO$ with $\overline{\cO} \subset \partial G \setminus \{ (0,0) \}$. This yields that the boundary condition $u \cdot \nu =0$ is fulfilled in a local sense away from $0$.

We define the Stokes operator $A_S$  
as the part of $A_p$ in $L^p_\sigma(G)$, i.e.,
\begin{equation}\label{defstokesop}
	\begin{aligned}
		A_S u &:= A_p |_{L^p_\sigma(G)} u ,  \ \ u \in D(A_S), \\
		D(A_S) & := \left\{ v \in D(A_p) \cap L^p_\sigma(G): \ A_p v \in L^p_\sigma(G) \right\}.
	\end{aligned}
\end{equation} 
Note that for the boundary conditions considered here, as a well-known
fact, the Helmholtz projection and the Laplacian commute, and the projection does not appear in the definition of $A_S$. The next lemma justifies this definition of the Stokes operator.
\begin{lemma} \label{lm:4.1} Let $1<p<\infty$. Then
\begin{equation*} D(A_{S})=D(A_p)\cap L^p_{\sigma}(G).
\end{equation*}
\end{lemma}
\begin{proof} We only have to show, that the right-hand side is a 
subset of $D(A_{S})$. To this end, let $u\in D(A_p)\cap
L^p_{\sigma}(G)$ and $f:=A_p u$. 
It remains to show that $f\in L^p_{\sigma}(G)$. 
By the fact that $f=A_pu=\curl'\curl u$
	and $u\in D(A_p)\cap L^p_{\sigma}(G)$, the Gau{\ss} theorem
	yields
	\begin{align*}  
	\int_{G}f\cdot \nabla \vp\, d(x_1,x_2)
	&=\int_{G}(\curl'\curl u)\cdot \nabla \vp\, d(x_1,x_2)\\
	&= -\left\langle\curl u,\, \nu\cdot \curl'\vp\right
	    \rangle_{W^{1-1/p}_p(\partial G), W^{-1/p'}_{p'}(\partial G)}
	=0 
	\end{align*} 
	for all $\vp \in\widehat{W}^{1,p'}(G)$.
Note that $\mathrm{div}\,\curl'\vp=0$, hence the trace $\nu\cdot \curl'\vp$
is defined in $W^{-1/p'}_{p'}(\partial G)$ in the usual sense.
By the fact that $\curl u\in W^{1,p}(G,\R^2)$ therefore the
duality pairing on the boundary above is well-defined.
The proof is complete.
\end{proof}

Recall from \eqref{lemm_projektion} that $L^p(\Omega,\R^2)$ is
decomposed in $L^p(\R,E^p_3)$ and $L^p(\mathbb{R},\la e_0,e_1,
e_2\ra)$ with $E^p_m$ defined in the lines before \eqref{lemm_projektion}
and $e_0,e_1,e_2$ the normed eigenfunctions to the first three eigenvalues
of the operator $\mathcal T_{p,\theta}$ introduced
in Subsection~\ref{sec_prob_transformiert}.

In order to transfer the properties of $\mA_p$ to the 
Stokes operator $A_S$ a crucial point is that 
$\widetilde\Theta_*^pL^p(\mathbb{R},\la e_0,e_1,e_2\ra)$ does
not contain non-trivial solenoidal vector fields.
To carry over full Sobolev regularity, however, this fact is not 
enough. This purpose requires stronger properties:
\begin{proposition}\label{closedrangeq}
Let $1 <p< \infty$. Then there exists a $\delta>0$ such that
\begin{enumerate}
\item $\|\widetilde \mQ u\|_p\ge \delta\|u\|_p$ for all $u\in
L^p_\sigma(G)$,
\item $\|\mQ u\|_{K_p^2}\ge \delta\|u\|_{K_p^2}$ 
for all $u\in D(B_p)$ such that $\div\, u=0$, and 
\item $\|\mQ u\|_{D(A_p)}\ge \delta\|u\|_{D(A_p)}$ for all $u\in D(A_S)$.
\end{enumerate}
\end{proposition}
\begin{remark}
Proposition~\ref{closedrangeq} relies of course on the specific structure
of the solenoidal subspace. In fact, its proof (including the proof of the 
subsequent Lemma~\ref{divopbddhat}) shows that the operator '$\div$' is
isomorphic on the complemented space
to $\mL^p$ and on the corresponding higher order complemented 
subspaces. Furthermore, it keeps the 
complemented structure in its image. This essentially can be read off the
representations of the transformed '$\div$' operator applied on elements of
the complemented subspaces given in
(\ref{divvpcalc}) and (\ref{divvpcalcot}) below.
\end{remark}
\begin{proof}[Proof of Lemma~\ref{closedrangeq}(1)]
{\bf Step~1.}
Recall from Subsection~\ref{sec_prob_transformiert} that the eigenfunctions 
to the first three eigenvalues $(\lambda_i)_{i \in
\{0,1,2\}} \in \sigma(\mathcal T_{p,\theta})$ are explicitly given as
	\begin{itemize}
	\item $e_0(\theta):= \frac{1}{\sqrt{\theta_0}} \begin{pmatrix} 1 \\ 0
	\end{pmatrix} $ which corresponds to $\lambda_0= -1$  and
	\item $e_k(\theta) := \frac{1}{\sqrt{\theta_0}}  \begin{pmatrix} \cos
	(\frac{k \pi}{\theta_0} \theta )\\ - \sin( \frac{k
	\pi}{\theta_0} \theta ) \end{pmatrix} $ which corresponds to
	$\lambda_k:= -(\frac{k \pi}{\theta_0}-1)^2$ for $k \in \{1,2 \}$.
	\end{itemize}	
We notice that, depending on the value of the angle $\theta_0$, 
there might be a doubled eigenvalue. This, however, does not matter
for what follows.
An element $ \varphi\in L^p(\mathbb{R},\la e_0,e_1,e_2\ra)$ 
is then represented by
\begin{equation}\label{lincombvp} 
\varphi(x, \theta)
= \varphi_0(x)e_0(\theta)+\varphi_1(x)e_1(\theta)
+\varphi_2(x)e_2(\theta)
\end{equation}
with coefficients $\varphi_i \in L^p(\mathbb{R})$ for $i \in \{0,1,2\}$. 

{\bf Step~2.}
On our way to show (1) we first derive suitable estimates for 
$\vp\in L^p(\R,\la e_0,e_1,e_2\ra)$ in terms of the
transformed divergence operator. To this end, first observe that 
\[
	\mathrm{div}\, \widetilde\Theta_*^p v\circ\psi
	=e^{(\beta_p-3)x}\left((\beta_p-1+\partial_x)v_x+\partial_\theta
	v_\theta\right)=:e^{(\beta_p-3)x}\div_{\widetilde\Theta}\,v.
\]
Applying the transformed divergence operator $\div_{\widetilde\Theta}$ 
to representation (\ref{lincombvp}) yields
\begin{equation}\label{divvpcalc}
\begin{split}
 \div_{\widetilde\Theta}\,\vp
 & =\left(\beta_p -1+\partial_x\right)\varphi_0b_0
 + \left( \beta_p -1 -
 \frac{\pi}{\theta_0}+\partial_x \right) \varphi_1b_1 \\
 &\quad + \left( \beta_p -1 - \frac{2
 \pi}{\theta_0}+\partial_x \right)\varphi_2b_2  
\end{split} 
\end{equation}
where
\begin{equation}\label{defb02}
	\{b_0,b_1,b_2\}:=
	\left\{\frac1{\sqrt{\theta_0}},\,
	\frac{\cos(\frac{\pi}{\theta_0}\cdot)}{\sqrt{\theta_0}},\,
	\frac{\cos(\frac{2\pi}{\theta_0}
	\cdot)}{\sqrt{\theta_0}}\right\}
\end{equation}
is linearly independent in $L^p(I,\R)$. 
We set 
\[
	F_3^p:=\la b_0,b_1,b_2\ra. 
\]
The form of the coefficients in (\ref{divvpcalc}) is
\[
	(s_j+\partial_x)\vp_j,\quad s_j\in\R,\ j=0,1,2.
\]
Observe that depending on the values of $p$ and $\theta_0$
it can occur $s_j=0$. Thus,
in order to estimate expression (\ref{divvpcalc}) by 
$\vp_j$ from below we distinguish two cases: $s_j\neq 0$
for all $j=0,1,2$ or $s_j=0$ for one $j\in\{0,1,2\}$.

{\bf Step~2.1.} The case $s_j\neq 0$ for all $j=0,1,2$. Then we have 
\begin{equation}\label{isosjn0}
	s_j+\partial_x\in \sLis(L^p(\R),W^{-1,p}(\R)).
\end{equation}
Furthermore, since $F^{p'}_3$ is finite dimensional, we observe that 
$W^{1,p'}(\R,F^{p'}_3)$ is isomorphic to the space
\[
	W^{1,p'}(\R,F^{p'}_3)\cap L^{p'}\bigl(\R,W^{1,p'}(I,\R)\bigr).
\]
This implies that the norm of $W^{1,p'}(\R,F^{p'}_3)$ 
and the norm of $W^{1,p'}(\Omega,\R)$
are equivalent on $W^{1,p'}(\R,F^{p'}_3)$ and that
the latter space can be regarded as a closed subspace of 
$W^{1,p'}(\Omega,\R)$. Utilizing these facts, we can estimate as
\begin{align*}
	\|\vp_j\|_p&\le C\|(s_j+\partial_x)\vp_j\|_{W^{-1,p}(\R)}\\
	&\le C\bigl\|\sum_{j=0}^2(s_j+\partial_x)\vp_jb_j
	     \bigr\|_{W^{-1,p}(\R,F^p_3)}
	=C\|\div_{\widetilde\Theta}\vp\|_{W^{-1,p}(\R,F^{p}_3)}\\
	&=C \sup_{0\neq h\in W^{1,p'}(\R,F^{p'}_3)}
	    \frac{|\langle h,\,\div_{\widetilde\Theta}\vp
	    \rangle|}{\|h\|_{W^{1,p'}(\R,F^{p'}_3)}}\\
	&\le C \sup_{0\neq h\in W^{1,p'}(\Omega,\R)}
	    \frac{|\langle h,\,\div_{\widetilde\Theta}\vp
	    \rangle|}{\|h\|_{W^{1,p'}(\Omega,\R)}}
	    =\|\div_{\widetilde\Theta} \vp\|_{W^{-1,p}_0(\Omega,\R)},
\end{align*}
for $j=0,1,2$ with $C>0$ independent of $\vp$ and where
$W^{-1,p}_0(\Omega,\R)=(W^{1,p'}(\Omega,\R))'$.

{\bf Step~2.2.} The case $s_\ell=0$ for one $\ell\in\{0,1,2\}$. 
This case is more involved, since here we have 
\[
	s_\ell+\partial_x=\partial_x\in \sLis(L^p(\R),\hW^{-1,p}(\R)),
\]
whereas for the remaining $j\in \{0,1,2\}\setminus\{\ell\}$
we still have (\ref{isosjn0}). We set 
\begin{equation}\label{defujsp}
	U_j:=
	\left\{
	\begin{array}{rl}
		\hW^{-1,p}(\R,\la b_j\ra),& \text{if}\ j=\ell,\\
		W^{-1,p}(\R,\la b_j\ra),& 
		\text{if}\ j\in \{0,1,2\}\setminus\{\ell\},
	\end{array}
	\right.
\end{equation}
and
\begin{equation}\label{defvsp}
	V:=\overline{\div_{\widetilde\Theta}
	L^p(\R,E^p_3)}^{W^{-1,p}(\Omega,\R)}.
\end{equation}
In Lemma~\ref{divopbddhat} below it is proved that 
the sum of $U_0\oplus U_1\oplus U_2$ and $V$ is direct 
and consequently that 
\[
	U_0\oplus U_1\oplus U_2\oplus V,
	\quad
	\|\cdot\|_{U_0\oplus U_1\oplus U_2\oplus V}
	:=\|\cdot\|_{U_0}
	+\|\cdot\|_{U_1}
	+\|\cdot\|_{U_2}
	+\|\cdot\|_{V}
\]
is a Banach space. Then, this time we obtain 
\begin{align*}
	\|\vp_j\|_p&\le C\|(s_j+\partial_x)\vp_jb_j\|_{U_j}\\
	&\le C\bigl\|\sum_{j=0}^2(s_j+\partial_x)\vp_jb_j
	     \bigr\|_{U_0\oplus U_1\oplus U_2}
	=C\|\div_{\widetilde\Theta}\vp\|_{U_0\oplus U_1\oplus U_2}\\
	&\le C\|\div_{\widetilde\Theta} \vp\|_{U_0\oplus U_1\oplus U_2
	      \oplus V}
\end{align*}
for $j=0,1,2$ with $C>0$ independent of $\vp$.

{\bf Step~3.} 
Now, let $u\in L^p_\sigma(G)$ and $\vp \in
L^p(\mathbb{R},\la e_0,e_1,e_2\ra)$ 
such that $\widetilde\mQ u=u-\widetilde\Theta^p_*\vp$.
Observe that both,
\[
	\div_{\widetilde\Theta}:L^p(\Omega,\R^2)\to W^{-1,p}_0(\Omega,\R)
\]
and by Lemma~\ref{divopbddhat} also
\[
	\div_{\widetilde\Theta}:L^p(\Omega,\R^2)
	\to U_0\oplus U_1\oplus U_2\oplus V
\]
are bounded operators.
By the fact that $\div_{\widetilde\Theta} \widetilde\Theta^*_p u=0$,
we can continue the calculations in steps~2.1 and 2.2 
to the result that
\begin{equation*}\label{1estvpj}
\begin{split}
	\|\vp_j\|_p
	&\le C\|\div_{\widetilde\Theta}\,\vp\|_{\mathcal W}
	= C\|\div_{\widetilde\Theta}(\widetilde\Theta^*_p u-
	\vp)\|_{\mathcal W}\\
	&\le C\|u-\widetilde\Theta^p_*\vp\|_{L^p(G,\R^2)}
	=C\|\widetilde\mQ u\|_p\quad (j=0,1,2),
\end{split}
\end{equation*}
where $\mathcal W$ denotes either the space $W^{-1,p}_0(\Omega,\R)$
or the space $U_0\oplus U_1\oplus U_2\oplus V$, depending on whether 
we have $s_j\neq 0$ for all $j$ or $s_j=0$ for one $j$.
Summing up over $j$ yields
\[
	\|\vp\|_{p}
	=\|\vp\|_{L^p(\R,\la e_0,e_1,e_2\ra)}
	\le C\sum_{j=0}^2\|\vp_j\|_p
	\le C\|\widetilde\mQ u\|_p
\]
for all $u\in L^p_\sigma(G)$ and 
$\widetilde\Theta^p_*\vp=(1-\widetilde\mQ)u$. By the fact that 
\[
	C_0\|\vp\|_{L^p(\Omega,\R^2)}
	\ge\|\widetilde\Theta^p_*\vp\|_{L^p(G,\R^2)}
	=\|u-\widetilde\mQ u\|_p
	\ge \|u\|_p-\|\widetilde\mQ u\|_p
\]
we arrive at (1) by setting $\delta:=1/(C_0C+1)$. 

{\em Proof of (2).} The proof of (2) is in large parts similar to
the proof of (1). Hence we will be briefer in detail.

{\bf Step~1.}
Again we will first provide estimates for 
$\vp\in (1-\mP_3)\Theta^*_pD(B_p)$ in terms of the transformed 
divergence. Note that such a $\vp$ is still 
represented by (\ref{lincombvp}), but now with coefficients
$\vp_j\in W^{2,p}(\R)$.
The transformed divergence operator here is
\[
	\mathrm{div}\, \Theta_*^p v\circ\psi
	=e^{(\beta_p-1)x}\left((\beta_p+1+\partial_x)v_x+\partial_\theta
	v_\theta\right)=:e^{(\beta_p-1)x}\div_{\Theta}\,v.
\]
Consequently, 
\begin{equation}\label{divvpcalcot}
\begin{split}
 \div_{\Theta}\,\vp
 & =\left(\beta_p +1+\partial_x\right)\varphi_0b_0
 + \left( \beta_p +1 -
 \frac{\pi}{\theta_0}+\partial_x \right) \varphi_1b_1 \\
 &+ \left( \beta_p +1 - \frac{2
 \pi}{\theta_0}+\partial_x \right)\varphi_2b_2  
\end{split} 
\end{equation}
for $\vp\in (1-\mP_3)\Theta^*_pD(B_p)$. Again we write
the coefficients as $(s_j+\partial_x)\vp_j$. Here still 
$s_1$ and $s_2$ can vanish. Hence we again distinguish 
the two cases: $s_j\neq 0$
for all $j=0,1,2$ or $s_j=0$ for one $j\in\{1,2\}$.

{\bf Step~1.1.} For the case $s_j\neq 0$ for all $j=0,1,2$ we use
\[
	s_j+\partial_x\in \sLis(W^{2,p}(\R),W^{1,p}(\R))
\]
in order to deduce
\begin{align*}
	\|\vp_j\|_{W^{2,p}(\R)}
	&\le C\|(s_j+\partial_x)\vp_j\|_{W^{1,p}(\R)}\\
	&\le C\bigl\|\sum_{j=0}^2(s_j+\partial_x)\vp_jb_j
	     \bigr\|_{W^{1,p}(\R,F^p_3)}
	\le C\|\div_{\Theta} \vp\|_{W^{1,p}(\Omega,\R)}
\end{align*}
for $j=0,1,2$ with $C>0$ independent of $\vp$.

{\bf Step~1.2.}  If $s_\ell=0$ for one $\ell\in\{1,2\}$ we use
for that $\ell$,
\[
	s_\ell+\partial_x=\partial_x\in \sLis(\hW^{2,p}(\R),\hW^{1,p}(\R))
\]
to estimate
\begin{align*}
	\|\vp_\ell\|_{\hW^{2,p}(\R)}
	&\le C\|(s_\ell+\partial_x)\vp_\ell\|_{\hW^{1,p}(\R)}
	\le C\|(s_\ell+\partial_x)\vp_\ell\|_{W^{1,p}(\R)}\\
	&\le C\bigl\|\sum_{j=0}^2(s_j+\partial_x)\vp_jb_j
	     \bigr\|_{W^{1,p}(\R,F^p_3)}
	\le C\|\div_{\Theta} \vp\|_{W^{1,p}(\Omega,\R)}
\end{align*}
with $C>0$ independent of $\vp$.
The corresponding estimate for $\vp$ in the $L^p$-norm can be 
established completely analogous as in step~2.2 of the proof of (1).
In this regard, observe that all assertions there as well as of 
Lemma~\ref{divopbddhat} obviously remain true, if we replace
$\div_{\widetilde\Theta}$ by $\div_{\Theta}$. Hence we obtain
\begin{align*}
	\|\vp_\ell\|_{L^p(\R)}
	&\le C\|\div_{\Theta} \vp\|_{U_0\oplus U_1\oplus U_2
	      \oplus V}.
\end{align*}
Taking into account the well-known interpolation estimate
$\|\nabla v\|_{L^p(\R)}
\le C(\|\nabla^2 v\|_{L^p(\R)}+\|v\|_{L^p(\R)})$, altogether 
we have
\begin{align*}
	\|\vp_j\|_{W^{2,p}(\R)}
	&\le C\left(\|\div_{\Theta} \vp\|_{W^{1,p}(\Omega,\R)}
	+\|\div_{\Theta} \vp\|_{U_0\oplus U_1\oplus U_2
	      \oplus V}\right)
\end{align*}
for $j=0,1,2$ with $C>0$ independent of $\vp$.

{\bf Step~2.} Let $u\in D(B_p)$ with $\div\, u=0$ and 
$\vp \in (1-\mP_3)\Theta^*_pD(B_p)$ 
such that $\mQ u=u-\Theta^p_*\vp$. Thanks to Lemma~\ref{divopbddhat}
and since
\[
	\div_{\Theta}:W^{2,p}(\Omega,\R^2)\to W^{1,p}(\Omega,\R)
\]
is bounded, by virtue of $\div_\Theta \Theta^*_pu=0$ and 
the estimates in Steps~1.1 and 1.2 we conclude
\begin{equation*}\label{1estvpjda}
\begin{split}
	\|\vp_j\|_{W^{2,p}(\R)}
	&\le C\left(\|\div_{\Theta}\,\vp\|_{W^{1,p}(\Omega,\R)}
	+\|\div_{\Theta} \vp\|_{U_0\oplus U_1\oplus U_2
	      \oplus V}\right)\\
	&\le C\|\Theta^*_p u-\vp\|_{W^{2,p}(\Omega,\R)}\\
	&\le C\|u-\Theta^p_*\vp\|_{K_p^2(G,\R^2)}
	=C\|\mQ u\|_{K_p^2(G,\R^2)}\quad (j=0,1,2).
\end{split}
\end{equation*}
Summing up over $j$, analogous to step~3 of the proof of (1) we arrive at
(2). 

\medskip
{\em Proof of (3).} According to Lemma~\ref{lempartap}(1), 
$\|\cdot\|_p+\|\cdot\|_{K^2_p}$ is an equivalent norm on $D(A_p)$
and we have $\mQ=\widetilde\mQ$ on $D(A_p)$.
The estimates proved in (1) and (2) then yield 
\begin{align*}
	\|u\|_{D(A_p)}&\le C\left(\|u\|_p+\|u\|_{K^2_p}\right)
	\le C\left(\|\mQ u\|_p+\|\mQ u\|_{K^2_p}\right)\\
	&\le C\|\mQ u\|_{D(A_p)}
	\quad (u\in D(A_S)).
\end{align*}
The proof is now completed.
\end{proof}

We have used the following facts in the proof of 
Proposition~\ref{closedrangeq}.
\begin{lemma}\label{divopbddhat}
Let $1 <p< \infty$. Let $U_j$, $j=0,1,2$, $\div_{\widetilde\Theta}$, and
$V$ be as defined in the proof of Proposition~\ref{closedrangeq}(1).
Then $U_0,U_1,U_2,V$ are Banach spaces, their sum is direct, and
we have
\begin{equation}\label{bdddivsum}
	\div_{\widetilde\Theta}\in\sL\left(L^p(\Omega,\R^2),\,
	U_0\oplus U_1\oplus U_2\oplus V\right).
\end{equation}
\end{lemma}
\begin{proof}
By their definition (\ref{defujsp}) and (\ref{defvsp}) it is obvious that 
$U_0,U_1,U_2,V$ are Banach spaces and that the sum of $U_0,U_1,U_2$
is direct. Note that 
\[
	L^p(\Omega,\R^2)=L^p(\R,E^p_3)
	\oplus L^p(\R,\la e_0\ra)
	\oplus L^p(\R,\la e_1\ra)
	\oplus L^p(\R,\la e_2\ra).
\]
It is also obvious that 
$\div_{\widetilde\Theta}:L^p(\R,\la e_j\ra)\to U_j$ and hence also
\begin{equation}\label{bdddivvc}
	\div_{\widetilde\Theta}:L^p(\R,\la e_0,e_1,e_2\ra)
	\to U_0\oplus U_1\oplus U_2
\end{equation}
is bounded (even isomorphic due to the estimates for $\vp$ in
steps~2.1 and 2.2 of the proof of Proposition~\ref{closedrangeq}). Due to 
$\div_{\widetilde\Theta}\in\sL\bigl(L^p(\Omega,\R^2),W^{-1,p}(\Omega,\R)\bigr)$ 
we see that by definition of $V$ the operator 
\begin{equation}\label{bdddivv}
	\div_{\widetilde\Theta}:L^p(\R,E^p_3)\to V
\end{equation}
is bounded too.
It remains to prove that the sum of $V$ and $U_0\oplus U_1\oplus U_2$ 
is direct. 

To this end, denote by 
$\cQ_3:W^{1,p'}(\Omega,\R)\to W^{1,p'}(\Omega,\R)$ the projector 
\[
	\cQ_3 v:=\sum_{j=0}^2(v,b_j)b_j,\quad v\in W^{1,p'}(\Omega,\R)
\]
with $b_j$, $j=0,1,2$, be defined as in (\ref{defb02}). 
Writing
\[
	W^{1,p'}(\Omega,\R)
	=W^{1,p'}(\R,L^{p'}(I,\R))\cap L^{p'}(\R,W^{1,p'}(I,\R))
\]
it is easily seen that $\cQ_3$ is a bounded projector
onto $W^{1,p'}(\R,F^{p'}_3)$. 
Note that $(b_k)_{k=0}^\infty$ with 
$b_k(\theta)=\cos(k\pi\theta/\theta_0)/\sqrt{\theta_0}$ 
as the collection of eigenfunctions of the
Neumann-Laplacian on the interval $I=(0,\theta_0)$ forms 
an orthonormal Hilbert basis of $L^2(I,\R)$. This shows that $\cQ_3$ is
symmetric, hence $\cQ_3$ is a bounded projector on
\[
	W^{-1,p}_0(\Omega,\R)
	=(W^{1,p'}_0(\Omega,\R))'
	=W^{-1,p}(\R,L^p(I,\R))+ L^{p}(\R,W^{-1,p}_0(I,\R)),
\]
too. Since all norms on $F^p_3$ are equivalent, for its image we calculate
\begin{equation}\label{projqw-1}
\begin{split}
	\cQ_3W^{-1,p}_0(\Omega,\R)
	&=W^{-1,p}(\R,F^p_3)+ L^{p}(\R,F^p_3)
	=W^{-1,p}(\R,F^p_3)\\
	&=W^{-1,p}(\R,\la b_0\ra)\oplus W^{-1,p}(\R,\la b_1\ra)
	  \oplus  W^{-1,p}(\R,\la b_2\ra).
\end{split}
\end{equation}

We next show that $V\subset (1-\cQ_3)W^{-1,p}_0(\Omega,\R)$.
By the fact that $(e_k)_{k=0}^\infty$ forms a basis of $L^2(I,\R^2)$
(see (\ref{spectt}) and the subsequent lines),
every $v\in L^2(\R,E^2_3)$ is represented as 
$v=\sum_{k=3}^\infty v_ke_k$ with $(v_k)\subset L^2(\R)$.
Hence we obtain
\begin{align*}
	\div_{\widetilde\Theta}v
	&=\sum_{k=3}^\infty 
	(\beta_2-1+\partial_x)v_k e_k^1
	+v_k \partial_\theta e_k^2\\
	&=\sum_{k=3}^\infty 
	\left(\beta_2-1\pm\frac{k\pi}{\theta_0}
	+\partial_x\right)v_k b_k.
\end{align*}
This shows that
\[
	\cQ_3\div_{\widetilde\Theta}v=0
	\quad \bigl(v\in L^p(\R,E^p_3)\cap L^2(\R,E^2_3)\bigr).
\]
The boundedness of the operators $\div_{\widetilde\Theta},\cQ_3$ 
and a density argument yield that this identity
remains true for all $v\in L^p(\R,E^p_3)$. Once more the 
boundedness of $\cQ_3$ on $W^{-1,p}_0(\Omega,\R)$
then gives $V\subset (1-\cQ_3)W^{-1,p}_0(\Omega,\R)$.

Finally, $W^{1,p}(\R,\la b_j\ra)\hookd \hW^{1,p}(\R,\la b_j\ra)$ 
implies
\[
	\hW^{-1,p}(\R,\la b_j\ra)\hook W^{-1,p}(\R,\la b_j\ra).
\]
In combination with (\ref{projqw-1}) this gives 
\[
	U_0\oplus U_1\oplus U_2\subset \cQ_3 W^{-1,p}_0(\Omega,\R),
\]
hence $V\cap \bigl(U_0\oplus U_1\oplus U_2\bigr)=\{0\}$.

Since we equip $U_0\oplus U_1\oplus U_2\oplus V$ with the norm
$\|\cdot\|_{U_0\oplus U_1\oplus U_2\oplus V}
	:=\|\cdot\|_{U_0}
	+\|\cdot\|_{U_1}
	+\|\cdot\|_{U_2}
	+\|\cdot\|_{V}$,
relations (\ref{bdddivvc}) and (\ref{bdddivv}) result in
(\ref{bdddivsum}). Now all assertions are proved.
\end{proof}

\begin{corollary}\label{closedrangeqcor}
Let $1 <p< \infty$. Then we have that
\begin{enumerate}
\item $\widetilde \mQ L^p_\sigma(G)$ is closed in
$\mL^p$ and $\widetilde\mQ\in \sLis\left(L^p_\sigma(G),\,
\widetilde \mQ L^p_\sigma(G)\right)$,
\item $\mQ D_\sigma$ is closed in
$D(\mB_p)$ and $\mQ\in \sLis\left(D_\sigma,\,\mQ D_\sigma\right)$,
where $D_\sigma:=\{v\in D(\mB_p):\ \div\, v=0\}$, and 
\item $\mQ D(A_S)$ is closed in
$D(\mA_p)$ and $\mQ\in \sLis\left(D(A_S),\,\mQ D(A_S)\right)$.
\end{enumerate}
\end{corollary}
With these facts at hand we can prove our main result on the 
Stokes operator.
\begin{proof}[Proof of Theorem~\ref{main2}]
Assume that $\lambda\in\rho(A_p)$.
By the fact that $A_S$ is the part of $A_p$ from Lemma~\ref{lm:4.1}
we infer that 
\[
	(\lambda-A_S)^{-1}
	=(\lambda-A_p)^{-1}|_{L^p_\sigma(G)}.
\]
In combination with Lemma~\ref{lempartap}(3),(4) this implies
\[
	\mQ(\lambda-A_S)^{-1}u=(\lambda-\mA_p)^{-1}\widetilde\mQ u
	\quad (u\in D(A_S)).
\]
In particular, the above line yields $(\lambda-\mA_p)^{-1}\widetilde\mQ
L^p_\sigma(G)\subset \mQ D(A_S)$.
Thus, thanks to Corollary~\ref{closedrangeqcor} we conclude that
\begin{equation}\label{crurepresas}
	(\lambda-A_S)^{-1}f
	=\mQ^{-1}(\lambda-\mA_p)^{-1}\widetilde\mQ f
	\quad (L^p_\sigma(G)).
\end{equation}

For $1<p<1+\delta$ with $\delta>0$ given in Theorem~\ref{thmmasa}
we know by that result that the resolvent set of $A_p$ contains
a suitable sector.
For those $p$ the assertion hence follows from 
Corollary~\ref{closedrangeqcor} and Theorem~\ref{mainlap}. 
For general $p\in(1,\infty)$ representation
(\ref{crurepresas}) gives a candidate for the resolvent of 
$A_S$. In fact, choosing $1<q<1+\delta$, on   
$L^p_\sigma(G)\cap L^q_\sigma(G)$ we already know that
it is the resolvent. A density argument and again 
Corollary~\ref{closedrangeqcor} and Theorem~\ref{mainlap}
then yield the assertion. 
\end{proof}
\begin{remark}\label{stokes_re_cons}
From Proposition~\ref{proppropq}(1) and 
Theorem~\ref{korollar_ergebnis_paper}
it also follows consistency of the resolvent of
$A_S$, that is, for every $\lambda\in\rho(A_S)$ the family
	$\left((\lambda-A_S)^{-1}\right)_{1<p<\infty}$ is consistent
	on the scale $\left(L^p_\sigma(G)\right)_{1<p<\infty}$.
\end{remark}

Finally we prove our third main result.
\begin{proof}[Proof of Theorem~\ref{korollarmainstokes}]
We follow the strategy in the proof of Theorem~\ref{main2}.
For $f\in L^p_\sigma(G)$ the candidate for the solution of  
\begin{equation}
	 	\left.
	 	\begin{array}{r@{\ =\ }lll}
	 		- \Delta u+\nabla \pi &  f &
			\text{in} & G, \\
	 		\div\, u&0&
			\text{in} & G, \\
			\text{curl}\, u=0, \ u \cdot \nu &0 &
			\text{on} & \partial G 
	 	\end{array}
	 	\right\}
	 \label{stokequ}
\end{equation}
is given as $\pi=0$ and $u=\mQ^{-1}\mB_p^{-1}\widetilde\mQ f$.
Thanks to Proposition~\ref{proppropq} and Corollary~\ref{closedrangeqcor} 
it remains to show that $\div\, u=0$. This, in turn, follows
from Proposition~\ref{mainlapconv}, 
$\mQ^{-1}(\lambda-\mA_p)^{-1}\widetilde\mQ f\subset D(A_S)$, 
and the fact that the operator
$\div$ acts continuously on the space $K_p^2(G,\R^2)$.
\end{proof}

\appendix

\section{Elements from harmonic and functional analysis}\label{appa}

The following facts might be well-known. Since we could not
find an appropriate reference, we give their proofs here.
\begin{lemma}\label{denselem}
Let $X,Y$ be Banach spaces such that $X\hook Y$. 
Then we have 
\[
	C^\infty_c(\R,X)\hookd W^{k,p}(\R,X)\cap W^{\ell,p}(\R,Y)
\]
for every $k,\ell\in\N_0$ and $p\in(1,\infty)$.
\end{lemma}
\begin{proof}
First recall that
\[
	C^\infty_c(\R,E)\hookd W^{k,p}(\R,E)
\]
for every $k\in\N_0$, $p\in(1,\infty)$, and arbitrary Banach space $E$.
In fact, it is standard to construct a (universal) sequence of operators
$(\Phi_k)_{k\in\N}$ such that for $u\in W^{k,p}(\R,E)$ we have 
$(\Phi_ku)\subset C^\infty_c(\R,E)$ and 
\[
	\Phi_ku\to u\quad\text{in } W^{k,p}(\R,E)\quad (k\to\infty)
\]
for every $k\in\N_0$, $p\in(1,\infty)$, and arbitrary Banach space $E$.
Since $X\subset Y$, for $u\in W^{k,p}(\R,X)\cap W^{\ell,p}(\R,Y)$
this gives $\Phi_ku\to u$ in $W^{k,p}(\R,X)$ and in
$W^{\ell,p}(\R,Y)$. 
\end{proof}

Let $T:D(T)\subset X\to X$ be a closed, densely defined operator on a 
Banach space $X$. We denote by 
\[
	T^\sharp:X'\to D(T)' 
\]
the dual operator of $T$, regarded as a bounded operator from $D(T)$
to $X$, and by
\[
	T':D(T')\subset X'\to X' 
\]
the usual Banach space dual operator of $T$. The fact that $D(T)\subset X$ is
dense, obviously implies $D(T')\hook X'\hook D(T)'$
and that
\begin{equation}\label{dualcons}
	T^\sharp|_{D(T')}= T'.
\end{equation}
Furthermore, we have the following lemma on consistency.
\begin{lemma}\label{conslem}
Let $X$ be a reflexive Banach space and let $T:D(T)\subset X\to X$ be 
densely defined such that $T\in\sLis(D(T),X)$. Assume there is an
embedding (with means i.p.\ injection) $J:D(T)\to X'$ with dense range. 
Then there exists an embedding 
$\widetilde J:X\to D(T)'$ such that, 
if $\widetilde J\circ T\subset T^\sharp\circ J$, we have
\begin{equation}\label{consequ}
	J\circ T^{-1}\circ\widetilde J^{-1}|_{\widetilde J X\cap X'}
	=(T^\sharp)^{-1}|_{\widetilde J X\cap X'}
	=(T')^{-1}|_{\widetilde J X\cap X'}
	\quad\text{in } X'.
\end{equation}
\end{lemma}
\begin{proof}
Since $\overline{D(T)}=X$ we have $X'\hook D(T)'$. Reflexivity of $X$  
and $J(D(T))\hookd X'$ further imply that there is an embedding
$\widetilde J:X\to D(T)'$.
Thus, $\widetilde J X\cap X'$ is well-defined and
due to $T\in\sLis(D(T),X)$ which also implies 
$T^\sharp\in\sLis(X',D(T)')$ and $T'\in\sLis(D(T'),X')$, 
line (\ref{consequ}) is meaningful.

Now, let $z\in \widetilde JX\cap X'$ and 
set $x_1:=JT^{-1}\widetilde J^{-1}z\in X'$ and
$x_2:=(T^\sharp)^{-1}z\in X'$. 
Thanks to $\widetilde J\circ T\subset T^\sharp\circ J$  
we obtain
\begin{align*}
	T^\sharp(x_1-x_2)
	&=T^\sharp\left(JT^{-1}\widetilde J^{-1}z-(T^\sharp)^{-1}z\right)\\
	&=\widetilde JTT^{-1}\widetilde J^{-1}z-T^\sharp(T^\sharp)^{-1}z=z-z=0.
\end{align*}
Thus $x_1=x_2$ in $X'$ and the assertion is proved.
The second equality in (\ref{consequ}) follows in a similar manner 
from (\ref{dualcons}).
\end{proof}


\begin{thebibliography}{10}

\bibitem{adams}
R.A. Adams and J.~J.~F. Fournier.
\newblock {\em {S}obolev {S}paces}.
\newblock Academic Press, 1975.

\bibitem{chua92}
S.-K. Chua.
\newblock Extension theorems on weighted {S}obolev spaces.
\newblock {\em Indiana Univ. Math. J.}, 41:1027--1076, 1992.

\bibitem{Coster-Nicaise: Helmholtz equation}
C. De Coster, S. Nicaise:
\newblock Singular behavior of the solution of the Helmholtz equation in weighted $L^p$-Sobolev spaces. 
\newblock {Advances Diff. Eq.} 16 (2011), 165-198.

\bibitem{Coster-Nicaise: Cauchy-Dirichtlet heat equation}
C. De Coster, S. Nicaise:
\newblock Singular behavior of the solution of the Cauchy-Dirichlet heat equation in weighted $L^p$-Sobolev spaces. 
\newblock {Bull. B.M.S.-Simon Stevin 11 (2011)}, 769-780.

\bibitem{dauge1989}
M.~Dauge.
\newblock Stationary {S}tokes and {N}avier-{S}tokes systems on two- or
  three-dimensional domains with corners. {I}. {L}inearized equations.
\newblock {\em SIAM J. Math. Anal.}, 20(1):74--97, 1989.

\bibitem{Denk-Hieber-Pruess:Maximal-Regularity}
R.~Denk, M.~Hieber, and J.~Pr{\"u}{\ss}.
\newblock {\em $\mathcal{R}$-{B}oundedness, {F}ourier-{M}ultipliers and
  {P}roblems of {E}lliptic and {P}arabolic {T}ype}, volume 166 of {\em Mem.
  Amer. Math. Soc.}
\newblock American Mathematical Society, 2003.

\bibitem{deuring1998}
P.~Deuring.
\newblock {$L^p$}-theory for the {S}tokes system in 3{D} domains with conical
  boundary points.
\newblock {\em Indiana Univ. Math. J.}, 47(1):11--47, 1998.

\bibitem{deuring2001}
P.~Deuring.
\newblock The {S}tokes resolvent in 3{D} domains with conical boundary
  points:non-regularity in {$L^p$}-spaces.
\newblock {\em Adv. Diff. Eq.}, 6:175--226, 2001.

\bibitem{Galdi2012}
G.P.~Galdi.
\newblock {\em An {I}ntroduction to the {M}athematical {T}heory 
of the {N}avier-{S}tokes {Equations}}.
\newblock Springer, 2011.

\bibitem{Grisvard}
P.~Grisvard.
\newblock {\em Elliptic {P}roblems in {N}onsmooth {D}omains}.
\newblock Society for Industrial and Applied Mathematics, 2011.

\bibitem{gs2006}
B.~Guo and C.~Schwab.
\newblock {Analytic regularity of Stokes flow on polygonal domains in countably
  weighted Sobolev spaces.}
\newblock {\em {J. Comput. Appl. Math.}}, 190(1-2):487--519, 2006.

\bibitem{Haase}
M.~Haase.
\newblock {\em The {F}unctional {C}alculus for {S}ectorial {O}perators}.
\newblock Operator Theory Advances and Applications. Springer, Basel, 2006.

\bibitem{hisa2016}
M.~Hieber and J.~Saal.
\newblock {\em The Stokes Equation in the {$L^p$}-setting: Well Posedness and
  Regularity Properties}, pages 1--88.
\newblock Handbook of Mathematical Analysis in Mechanics of Viscous Fluids.
  Springer, 2016.

\bibitem{jones81}
P.W. Jones.
\newblock Quasiconformal mappings and extendability of functions in {S}obolev
  spaces.
\newblock {\em Acta Math.}, 147:71--88, 1981.

\bibitem{Kaip}
M.~Kaip and J.~Saal.
\newblock The {P}ermanence of {$\mathcal{R}$}-{B}oundedness under
  {I}nterpolation and {A}pplications to {P}arabolic {S}ystems.
\newblock {\em J. Math. Sci. Univ. Tokyo}, 19:1--49, 2012.

\bibitem{Kalton-Weis:Operator-Sums}
N.~J. Kalton and L.~Weis.
\newblock {T}he ${H}^\infty$-{C}alculus and {S}ums of {C}losed {O}perators.
\newblock {\em Math. Ann.}, 321:319--345, 2001.

\bibitem{ko1976}
R.B. {Kellogg} and J.E. {Osborn}.
\newblock {A regularity result for the Stokes problem in a convex polygon.}
\newblock {\em {J. Funct. Anal.}}, 21:397--431, 1976.

\bibitem{kondra1967}
V.A. Kondrat'ev.
\newblock Boundary problems for elliptic equations in domains with conical or
  angular points.
\newblock {\em Trans. Mosc. Math. Soc}, 16:227--313, 1967.

\bibitem{Kozlov-Rossmann}
V. Kozlov, J. Rossmann:
\newblock On the nonstationary Stokes system in a cone. 
\newblock {J. Diff. Equ.} 260 (2016), 8277-8315.

\bibitem{Kunstmann}
P.~Kunstmann and L.~Weis.
\newblock {\em Maximal {$L_p$}-{R}egularity for {P}arabolic {E}quations,
  {F}ourier {M}ultiplier {T}heorems and {H$^\infty$}-{F}unctional {C}alculus.
  {I}n {F}unctional {A}nalytic {M}ethods for {E}volution {E}quations}.
\newblock volume 1855 of Lecture Notes in Mathematics, pages 65-311. Springer
  Berlin Heidelberg, 2004.

\bibitem{Maier}
S.~Maier and J.~Saal.
\newblock Stokes and {N}avier {S}tokes {E}quations with {P}erfect {S}lip on
  {W}edge {T}ype {D}omains.
\newblock {\em Discrete Contin. Dyn. Syst.- Series S}, 7(5):1045--1063, 2014.

\bibitem{maro2010}
V.G. Maz'ja and J.~Rossmann.
\newblock {\em Elliptic Equations in Polyhedral Domains}.
\newblock AMS, Providence, Rhode Island, 2010.

\bibitem{mitmon2008}
M.~Mitrea and S.~Monniaux.
\newblock The regularity of the {S}tokes operator and the {F}ujita-{K}ato
  approach to the {N}avier-{S}tokes initial value problem in {L}ipschitz
  domains.
\newblock {\em J. Funct. Anal.}, 254(6):1522--1574, 2008.

\bibitem{nauasaal}
T.~Nau and J.~Saal.
\newblock {$H^\infty$-{C}alculus for {C}ylindrical {B}oundary {V}alue
  {P}roblems.}
\newblock {\em Adv. Differ. Equ.}, 17(7-8):767--800, 2012.

\bibitem{Nazarov2001} A.I.~Nazarov.
\newblock {$L_p$}-estimates for a solution to the Dirichlet problem and to the Neumann problem for the heat equation in a wedge with edge of arbitrary
codimension.
{\em J. Math. Sci.}, 106:2989--3014, 2001.

\bibitem{Pruss}
J.~Pr\"uss and G.~Simonett.
\newblock H$^{\infty}$-{C}alculus for the {S}um of {N}on-{C}ommuting
  {O}perators.
\newblock {\em Trans. Amer. Math. Soc}, 359:3549--3565, 2007.

\bibitem{Pruss_Simonett}
J.~Pr\"uss and G.~Simonett.
\newblock {\em Moving Interface and Quasilinear Parabolic Evolution Equations}
\newblock { Monographs in Mathematics 105}, Birkh\"auser, (2016).

\bibitem{shen2012}
Z.~Shen.
\newblock Resolvent estimates in {$L_p$} for the {S}tokes operator in
  {L}ipschitz domains.
\newblock {\em Arch. Ration. Mech. Anal.}, 205(2):395--424, 2012.

\bibitem{Stein}
E.M. Stein.
\newblock {\em Singular Integrals and Differentiability Properties of
  Functions}.
\newblock Princeton University Press, 1970.

\bibitem{tolksdorf2017}
P.~Tolksdorf.
\newblock On the {$L^p$}-theory of the {N}avier-{S}tokes equations on
  three-dimensional bounded lipschitz domains.
\newblock arXiv:1703.01091.

\bibitem{triebel}
H.~Triebel.
\newblock {\em Interpolation {T}heory, {F}unction {S}paces, {D}ifferential
  {O}perators}.
\newblock North Holland, 1978.

\bibitem{voigt1992}
J.~Voigt.
\newblock Abstract {S}tein interpolation.
\newblock {\em Math. Nachr.}, 157:197--199, 1992.

\bibitem{Weis}
L.~Weis.
\newblock Operator-{V}alued {F}ourier {M}ultiplier {T}heorems and {M}aximal
  {L$_p$}-{R}egularity.
\newblock {\em Math. Ann.}, 2001.

\end{thebibliography}

\end{document}